\documentclass[11pt, a4paper, reqno]{amsart}
\usepackage[T1]{fontenc}
\usepackage[latin1]{inputenc}
\usepackage[english]{babel}
\usepackage[reqno]{amsmath}
\usepackage{amsthm}
\usepackage{latexsym}
\usepackage{amssymb}
\usepackage{mathrsfs}
\usepackage{mathtools}
\usepackage{mathabx}
\usepackage{lipsum}
\usepackage{titlesec}
\usepackage{tikz-cd}
\usepackage{dsfont} %identity operator \mathds{1}

\usepackage[totalwidth=15cm,totalheight=20cm, hmarginratio=1:1]{geometry}
\usepackage{mathtools}
\usepackage{color}
\usepackage{physics}
\usepackage{faktor}
\usepackage[shortlabels]{enumitem}
\usepackage{tikz} %for diagrams
\usepackage{hyperref}
\usepackage{graphicx}
\usepackage{tabularx}
\usepackage{cases} %cases with distinct labels
\newcolumntype{C}{>{\centering\arraybackslash}X} %colonna di tipo X con testo centrato
\usepackage{csquotes}
\usepackage{url}

\usepackage[style=alphabetic]{biblatex}

\usepackage{accents}

\makeatletter
\newtheorem*{rep@theorem}{\rep@title}
\newcommand{\newreptheorem}[2]{%
\newenvironment{rep#1}[1]{%
 \def\rep@title{#2 \ref{##1}}%
 \begin{rep@theorem}}%
 {\end{rep@theorem}}}
\makeatother

\makeatletter
\newtheorem*{rep@cor}{\rep@title}
\newcommand{\newrepcor}[2]{%
\newenvironment{rep#1}[1]{%
 \def\rep@title{#2 \ref{##1}}%
 \begin{rep@cor}}%
 {\end{rep@cor}}}
\makeatother

\makeatletter
\newtheorem*{rep@prop}{\rep@title}
\newcommand{\newrepprop}[2]{%
\newenvironment{rep#1}[1]{%
 \def\rep@title{#2 \ref{##1}}%
 \begin{rep@prop}}%
 {\end{rep@prop}}}
\makeatother

\newtheorem{theorem}{Theorem}[section]
\newreptheorem{theorem}{Theorem}
\numberwithin{theorem}{section}

\newtheorem{manualtheoreminner}{Theorem} % manually labelled theorem
\newenvironment{manualtheorem}[1]{%
    \renewcommand\themanualtheoreminner{#1}%
    \manualtheoreminner
}{\endmanualtheoreminner}

\newtheorem{manualcorollaryinner}{Corollary}
\newenvironment{manualcorollary}[1]{%
  \renewcommand\themanualcorollaryinner{#1}%
  \manualcorollaryinner
}{\endmanualcorollaryinner}

\newtheorem{lemma}[theorem]{Lemma}
\newtheorem{corollary}[theorem]{Corollary}

\newrepcor{corollary}{Corollary}

\newtheorem{proposition}[theorem]{Proposition}

\newrepprop{prop}{Proposition}

\theoremstyle{definition}
\newtheorem*{definition*}{Definition}
\newtheorem{definition}[theorem]{Definition}

\theoremstyle{remark}
\newtheorem{remark}[theorem]{Remark}

\makeatletter
\def\paragraph{\@startsection{paragraph}{4}%
  \z@\z@{-\fontdimen2\font}%
  {\normalfont\bfseries}}
\makeatother

\numberwithin{equation}{section}
\allowdisplaybreaks

\patchcmd{\subsection}{-.5em}{.5em}{}{}

\makeatletter

\renewcommand\section{\@startsection{section}{1}%
  \z@{.7\linespacing\@plus\linespacing}{.5\linespacing}%
  {\normalfont\scshape\centering}}

\renewcommand\subsection{\@startsection{subsection}{2}%
  \z@{-.5\linespacing\@plus-.7\linespacing}{.5\linespacing}%
  {\bfseries}}
  
\renewcommand\subsubsection{\@startsection{subsubsection}{3}%
  \z@{-.5\linespacing\@plus-.7\linespacing}{.5\linespacing}%
  {\itshape}}
  
\def\l@paragraph{\@tocline{4}{0pt}{1pc}{7pc}{}}

\makeatother

\newcommand{\C}{\mathbb{C}}
\newcommand{\R}{\mathbb{R}}

\newcommand{\Hyp}{\mathbb{H}^2}

\newcommand{\Teich}{\mathcal{T}}

\newcommand{\Lsl}{\mathfrak{sl}}

\newcommand{\vl}{|}
\newcommand{\Ker}{\mathrm{Ker}}

\newcommand{\PSL}{\mathbb{P}\mathrm{SL}}

\newcommand{\Isom}{\mathrm{Isom}}

\newcommand{\Id}{\mathrm{Id}}

\newcommand{\Ree}{\mathcal{R}e}

\renewcommand{\i}{\mathbf{I}}

\newcommand{\g}{\mathbf{g}}

\newcommand{\deft}{\mathcal{B}_0(T^2)}
\newcommand{\defp}{\mathcal{B}(T^2)}
\newcommand{\defg}{\mathcal{B}(\Sigma)}
\newcommand{\defgp}{\mathcal{B}_0(\Sigma)}
\newcommand{\RP}{\mathbb{R}\mathbb{P}^2}
\newcommand{\Sg}{\Sigma}

\newcommand{\ome}{\boldsymbol{\omega}}

\newcommand{\cubic}{Q^3(\Teich(T^2))}
\newcommand{\cubiccomplement}{Q^3_0(\Teich(T^2))}
\newcommand{\almost}{\mathcal{J}(\R^2)}

\newcommand{\dx}{\mathrm{d}x}
\newcommand{\dy}{\mathrm{d}y}
\newcommand{\du}{\mathrm{d}u}
\newcommand{\devu}{\mathrm{d}v}

\newcommand{\scalg}{\text{scal}}
\newcommand{\Xuno}{\mathbb{X}_{H_1}}
\newcommand{\Xdue}{\mathbb{X}_{H_2}}
\newcommand{\partialx}{\frac{\partial}{\partial x}}
\newcommand{\partialy}{\frac{\partial}{\partial y}}
\newcommand{\partialu}{\frac{\partial}{\partial u}}
\newcommand{\partialv}{\frac{\partial}{\partial v}}
\newcommand{\partialtheta}{\frac{\partial}{\partial\theta}}
\newcommand{\partials}{\frac{\partial}{\partial s}}
\newcommand{\partialHuno}{\frac{\partial}{\partial H_1}}
\newcommand{\partialHdue}{\frac{\partial}{\partial H_2}}
\newcommand{\partialbuno}{\frac{\partial}{\partial b_1}}
\newcommand{\partialbdue}{\frac{\partial}{\partial b_2}}

\newcommand{\partialHj}{\frac{\partial}{\partial H_j}}
\newcommand{\partialgi}{\frac{\partial}{\partial g_i}}

\newcommand{\partialbj}{\frac{\partial}{\partial b_j}}

\DeclarePairedDelimiterX{\scal}[2]{\langle}{\rangle}{#1 \mid #2}
\DeclarePairedDelimiterX{\scall}[2]{\langle}{\rangle}{#1, #2}

\DeclareMathOperator{\Imm}{Im}

\DeclareMathOperator{\End}{End}
\DeclareMathOperator{\Aut}{Aut}

\DeclareMathOperator{\SL}{\mathrm{SL}}

\DeclareMathOperator{\Proj}{\mathrm{Proj}}

\DeclareMathOperator{\Diff}{Diff}

\DeclareMathOperator{\Ad}{Ad}

\begin{document}

\setcounter{secnumdepth}{3}
\setcounter{tocdepth}{2}

\title[Global Darboux coordinates for complete Lagrangian fibrations%Symplectic and metric properties of properly convex \texorpdfstring{$\RP$}{RP2}-structures on \texorpdfstring{$T^2$}{T2}
]{Global Darboux coordinates for complete Lagrangian fibrations and an application to the deformation space of \texorpdfstring{$\RP$}{RP2}-structures in genus one%Symplectic and metric properties of properly convex \texorpdfstring{$\RP$}{RP2}-structures on the torus
}

\author[Nicholas Rungi]{Nicholas Rungi}
\address{NR: Scuola Internazionale Superiore di Studi Avanzati (SISSA), Trieste (TS), Italy.} \email{nrungi@sissa.it} 

\author[Andrea Tamburelli]{Andrea Tamburelli}
\address{AT: Dipartimento di Matematica, Universit\'a di Pisa, Italy.} \email{andrea.tamburelli@libero.it}

\date{\today}

\begin{abstract}
    In this paper we study a broad class of complete Hamiltonian integrable systems, namely the ones whose associated Lagrangian fibration is complete and has non compact fibres. By studying the associated complete Lagrangian fibration, we show that, under suitable assumptions, the integrals of motion can be taken as action coordinates for the Hamiltonian system. As an application we find global Darboux coordinates for a new family of symplectic forms $\ome_f$, parametrized by smooth functions $f:[0,+\infty)\to(-\infty,0]$, defined on the deformation space of properly convex $\RP$-structures on the torus. Such a symplectic form is part of a family of pseudo-K\"ahler metrics $(\g_f,\i,\ome_f)$ defined on $\deft$ and introduced by the authors. In the last part of the paper, by choosing $f(t)=-kt, k>0$ we deduce the expression for an arbitrary isometry of the space.
\end{abstract}

\maketitle

\tableofcontents

\section{Introduction}
An Hamiltonian system, from a geometric point of view, is a symplectic manifold $(M,\omega)$ of dimension $2n$ endowed with a smooth function $H:M\to\R$. It is called \emph{completely integrable} if there exist $n$ independent smooth functions $f_1:=H,f_2,\dots,f_n$ on the manifold $M$ with vanishing Poisson bracket $\{f_i,f_j\}=0, \forall i,j$ (\cite{abraham2008foundations}). %In the 19th century, Liouville was the first to find local solutions to a system of differential equations using the \emph{quadrature} method (\cite{liouville1855note}). Broadly speaking this consists in explicitly solving given equations and integrating given functions in one variable. In this way he obtained the existence of new $n$ functions $(\theta_1,\dots,\theta_n)$ such that the collection $\{f_1,\theta_1,\dots,f_n,\theta_n\}$ formed a local Darboux coordinates system.  \\ \\
It was proved by Arnlod (see \cite{arnold1968ergodic} and \cite{markus1974generic}), that smooth and compact components of the fibers $\{f_1=\text{constant},\dots,f_n=\text{constant}\}$ are diffeomorphic to $n$-dimensional tori and that one can always find, in a neighborhood of such fibers, \emph{angle-action coordinates}. These are coordinates $\{\theta_1,\psi_1,\dots,\theta_n,\psi_n\}$ which form a local Darboux frame for $\omega$ and the motion in the variables $\theta_1,\dots,\theta_n$ is \emph{quasi-periodic}. In particular, the integral of motions $f_1,\dots,f_n$ depends only on $\psi_1,\dots,\psi_n$ but, in general, they are different. \\ \\ 
Even though integrable Hamiltonian systems arise from a purely physical context, given the close connection with symplectic geometry, we are interested in understanding whether geometric spaces (e.g. deformation or moduli spaces) admit such a structure. Nowadays, the classic example is given by the so called \emph{Hitchin system} (\cite{hitchin1987stable}), that is the cotangent bundle of the moduli space of stable $G$-bundles on a compact algebraic curve, for $G$ a reductive Lie group. A much more recent result (\cite{choi2019symplectic}) showed that the Hitchin component of the $\SL(3,\R)$-character variety has the structure of an integrable Hamiltonian system with respect to the Goldman symplectic form (\cite{goldman1984symplectic}). In particular, it was given a geometric description of the construction in terms of convex $\RP$-structures.
\\ \\
On the other hand, a Lagrangian fibration $\pi:=(f_1,\dots,f_n):(M,\omega)\to\R^n$ is associated with every completely integrable Hamiltonian system $(M,\omega,f_1,\dots,f_n)$ (\cite{da2008lectures}). Moreover, in the case in which the Hamiltonian vector fields $\mathbb X_{f_1},\dots,\mathbb X_{f_n}$ are complete, the regular and connected fibers of $\pi$ are diffeomorphic to $\R^{k}\times T^{n-k}$, for $1\le k\le n$ (\cite{arnol2013mathematical}). Studying the topological and symplectic properties of this fibration, criteria have been found to be able to pick the integrals of motion $(f_1,\dots,f_n)$ as action variables, trying to give them a geometric meaning (see \cite{duistermaat1980global} for $k=0$ and \cite{choi2019symplectic} for $k=n$). \\ \\ 
In this paper, we study some properties of a new family of pseudo-K\"ahler metrics, introduced by the authors (\cite{rungi2021pseudo}), defined on the $\SL(3,\R)$-Hitchin component of the torus. Such family of metrics can be restricted to the deformation space of properly convex $\RP$-structures on the torus, which embeds as an open (but not closed) connected smooth manifold in the $\SL(3,\R)$-character variety. First, we show that the aforementioned deformation space has the structure of a complete Hamiltonian integrable system. Using the theory of complete Lagrangian fibrations (see Definition \ref{def:completefibration}), we show that two Hamilotonian functions associated with two group actions on the system can be taken as action variables. As a general result, we show that on a broad class of Hamiltonian systems the associated Lagrangian fibration is complete (see Proposition \ref{prop:completefibration}). Finally, we prove that, under suitable assumptions, the connected component of the identity of the isometry group is isomorphic to $\PSL(2,\R)\times S^1$ and, as a consequence, we deduce the expression of an arbitrary isometry of the space. %\\ \\ %We briefly recall the definition of a pseudo-K\"ahler metric on a smooth manifold and we state our main theorems.
\subsection{Symplectic properties}\label{sec:1.3}
Let $(M,\omega,H)$ be a completely integrable Hamiltonian system with integral of motions $f_1:=H,f_2,\dots,f_n$. By the well-know Arnold-Liouville theorem (see Theorem \ref{teo:arnlod}), the existence of the so-called \emph{action-angle} coordinates gives rise to a global Darboux frame $\{\theta_1,\psi_1,\dots,\theta_n,\psi_n\}$ for $\omega$ and a Lagrangian fibration $$\pi:=(f_1,\dots,f_n):(M,\omega)\to B\subset\R^n \ .$$Unfortunately, the action variables $\psi_1,\dots,\psi_n$ are not necessarily equal to the integral of motions $f_1,\dots,f_n$ and this does not allow, in general, to give a geometric description of the Darboux frame. \\ \\ A first general strategy to overcome this problem was presented in \cite{duistermaat1980global}, in the case in which the Lagrangian fibration $\pi:(M,\omega)\to B\subset\R^n$ associated with a complete Hamiltonian integrable system has fiber diffeomorphic to an $n$-dimensional torus. The crucial point is the existence of a global Lagrangian section $\sigma:B\to (M,\omega)$, which is guaranteed as long as $H^2(B,\R)\cong\{0\}$. \\ \\ Recently, Choi-Jung-Kim have presented an adapted version of this result, in the case where each fiber of $\pi:(M,\omega)\to B$ is diffeomorphic to $\R^n$ (see Theorem 3.4.5 in \cite{choi2019symplectic}). Their main application was the existence of a global Darboux frame for the Goldman symplectic form $\ome_G$ defined on the $\SL(3,\R)$-Hitchin component of a surface of genus at least two. \\ \\ In this paper, we are interested in studying completely integrable Hamiltonian systems whose associated Lagrangian fibration has non compact fibres. In order to get a similar result as the one in the compact or simply-connected case we need to focus on the so-called \emph{complete Lagrangian fibrations} (see Definition \ref{def:completefibration}). Such fibrations arise as the ones associated with a broad class of Hamiltonian systems, namely the ones with complete Hamiltonian vector fields $\mathbb X_{f_1},\dots,\mathbb X_{f_n}$ (see Proposition \ref{prop:completefibration}). Our main result shows that, in the non compact case, the integral of motions $f_1,\dots,f_n$ can be taken as action coordinates for $\omega$, whenever the base space $B$ is a contractible open subset of $\R^n$ and the associated Lagrangian fibration $\pi:(M,\omega)\to B$ is complete. Our strategy is adapted to that of the compact case (\cite{duistermaat1980global}), which is based on the existence of global Lagrangian sections and it is proved using the theory of sheaf cohomology (see Corollary \ref{cor:lagrangianglobalsection}). \\ \\
As an application, we study the symplectic geometry of $(\deft,\ome_f)$, which is a smooth manifold diffeomorphic to $\Hyp\times\C^*$ (Corollary \ref{cor:cubictorusandproperlyconvex}). There exists an action of $S^1$ and $\SL(2,\R)$ on $\deft$ which turns out to be Hamiltonian with respect to $\ome_f$ (Theorem \ref{teo:rungitamburelli2} and Theorem \ref{teo:rungitamburelli3}). In particular, by taking the action of the subgroup $\R^*<\SL(2,\R)$ generated by the diagonal matrices, it is possible to explicitly compute the Hamiltonian function $H_2$ with respect to this restricted action. Together with the Hamiltonian function $H_1$ of the circular action we get the existence of two commuting Hamiltonian vector fields $\Xuno,\Xdue$ on $\deft$. In other terms, the space $(\deft,\ome_f)$ has the structure of a complete Hamiltonian integrable system (see Proposition \ref{prop:deftintegrablesystem}). The main issue is that each fiber of the associated Lagrangian fibration $H:=(H_1,H_2):\big(\deft,\ome_f\big)\to B\subset\R^2$ is diffeomorphic to $\R\times S^1$ (see Remark \ref{rem:fiberofH}). However, since the base space $B$ is contractible and the Hamiltonian vector fields $\Xuno,\Xdue$ are complete, we can apply the theory developed in Section \ref{sec:completelagrangianfibrations} to get the following:
 \begin{manualtheorem}A \label{thmB}
The collection $\{\theta, H_1, s, H_2\}$ is a global Darboux frame for $\ome_f$, where $(s,\theta)\in\R\times S^1\cong H^{-1}(b)$ for each $b\in B$.
\end{manualtheorem}

\subsection{Pseudo-K\"ahler geometry}
%The aim of this short note is to study some differential geometric properties of a new family of pseudo-K\"ahler metrics, introduced by the authors (\cite{rungi2021pseudo}), on the deformation space $\deft$ of properly convex $\RP$-structures on the torus. \\ \\
A \emph{pseudo-Riemannian} metric $\g$ on a smooth $n$-manifold $M$ is an everywhere non-degenerate, smooth, symmetric $(0,2)$-tensor. The $\emph{index}$ of $\g$ is the maximal rank $k$ of the smooth distribution where it is negative-definite. For instance, if $k=0$ then $\g$ is a Riemannian metric. Now let $\mathbf{I}$ be a complex structure on $M$, then $(\g,\mathbf{I})$ is a \emph{pseudo-Hermitian structure} if \begin{equation*}
    \g(\mathbf{I}X,\mathbf{I}Y)=\g(X,Y), \qquad\forall X,Y\in T_pM, p\in M \ .
\end{equation*}Notice that, due to this last condition, the index of $\g$ in this case is always even $k=2s$, where $s$ is called the \emph{complex index} and it satisfies $1\le s\le m=\dim_\C M$. If $k=\dim_\C M$ we say the pseudo-metric has \emph{neutral signature}. The \emph{fundamental 2-form} $\ome$ of a pseudo-Hermitian manifold $(M,\g,\mathbf{I})$ is defined by: \begin{equation*}
    \ome(X,Y):=\g(X,\mathbf{I}Y), \qquad\forall X,Y\in T_pM, p\in M \ .
\end{equation*}
A pseudo-Hermitian manifold $(M,\g,\mathbf{I},\ome)$ is called \emph{pseudo-K\"ahler} if the fundamental $2$-form is closed, namely if $\mathrm{d}\ome=0$. In this case the corresponding metric is called \emph{pseudo-K\"ahler}.

\subsection{Metric properties}
The new family of pseudo-K\"ahler metrics $(\g_f,\i,\ome_f)$ is defined on the whole $\Hyp\times\C$ (see Section \ref{sec:3.1}) and it is parametrized by smooth functions $f:[0,+\infty)\to(-\infty,0]$ satisfying the following properties:$$
    (i) \ f(0)=0, \qquad (ii) \ f'(t)<0, \ \forall t\ge 0, \qquad (iii) \ \displaystyle\lim_{t\to+\infty}f(t)=-\infty \ .$$
In particular, they restrict to the hyperbolic K\"ahler metric on $\Hyp\times\{0\}$ (see Theorem \ref{teo:rungitamburelli}). \\ \\ Due to the complicated dependence of the metric and the symplectic form in terms of the function $f$, it is difficult to study basic properties of it. Nevertheless, as will be seen in Section \ref{sec:3}, exploiting the fact that a pseudo-K\"ahler metric with neutral signature enjoys the same symmetries as the positive-definite case, it is possible to greatly simplify the computation of the Ricci tensor and of the scalar curvature, even in the case of an arbitrary function $f$. Unfortunately, the final expression of these two quantities is slightly complicated. On the other hand, by choosing the functions $f(t)=-kt$ with $k>0$, we will obtain an estimate on the scalar curvature that will allow us to compute the connected component of the identity in the isometry group of the metric for this special case.\begin{manualtheorem}B \label{thmA}
If $f(t)=-kt$, with $k>0$, then any isometry $h:\big(\Hyp\times\C,\g_f\big)\to\big(\Hyp\times\C,\g_f\big)$ isotopic to the identity, can be written as $h=P\circ e^{i\theta}$, for some $P\in\PSL(2,\R)$ and $\theta\in\R$. 
\end{manualtheorem}
The Lie group $\Isom\big(\Hyp\times\C, \g_f\big)$ has four connected components. One is given by all the isometries preserving the orientations of both factors, one is given by all the isometries reversing the orientations of both factors (but preserving the orientation on the whole space) and the remaining two are formed by isometries reversing orientation on either $\Hyp$ or $\C$. In particular, the connected component of the identity is identified with those isometries preserving the orientations on both factors.
\begin{manualcorollary}C If $f(t)=-kt$, with $k>0$, then any isometry $h:\big(\Hyp\times\C,\g_f\big)\to\big(\Hyp\times\C,\g_f\big)$ not isotopic to the identity, can be written as $$h=P\circ e^{i\theta}\circ h_1,\quad h=P\circ e^{i\theta}\circ h_2,\quad h=P\circ e^{i\theta}\circ h_1\circ h_2\qquad P\in\PSL(2,\R),e^{i\theta}\in S^1$$ where $h_1(z,w):=(-\bar z, w)$ and $h_2(z,w):=(z,\bar w)$, according to which connected component it belongs.
\end{manualcorollary}
\section*{Acknowledgement}
Part of this work was done while the first author was visiting the second author at the University of Pisa. The second author acknowledges support from the U.S. National Science Foundation under grant DMS-2005501 and from the University of Pisa.
\section{Preliminaries}
\subsection{Background on properly convex \texorpdfstring{$\RP$}{RP2}-structures}

Here we give the standard definitions of \emph{convex} $\RP$-\emph{structures} over a smooth surface $S$ and we introduce the deformation space of such structures. For a more complete and detailed exposition see \cite{goldman1990convex}, \cite{choi1997classification} and \cite{goldman2018geometric}.\newline An $\RP$-\emph{structure} on a smooth connected surface $S$ is a maximal $\RP$-atlas, i.e. an atlas in which the local charts take value in the real projective plane and the transition functions are induced by projective transformations. Once a maximal $\RP$-atlas is given, we say that $S$ has an $\RP$-\emph{structure}. %By unravelling the definition it is easy to see that if $S$ is an $\RP$-surface and $p:\widetilde{S}\to S$ is its universal cover, then $\widetilde{S}$ inherits an $\RP$-structure from the one of $S$.
A domain (open and connected) $\Omega\subset\RP$ is said to be \emph{convex} if there exists a projective line $l$ disjoint from $\Omega$ such that $\Omega\subset\RP\setminus l\cong\mathbb{A}^2$ is convex in the usual sense. %By definition $\R^2$ is convex but $\RP$ is not. %It is not difficult to show that this notion of convexity in the real projective plane does not add new convex sets with respect to those usual one in affine spaces, see \cite[\S 1]{andersson2004complex} .
\begin{definition}
An $\RP$-surface $S$ is \emph{convex} if it is projectively isomorphic to a quotient $\Omega/_{\Gamma}$, where $\Omega\subset\RP$ is a convex domain and $\Gamma\subset\Proj(\Omega)\subset\SL(3,\R)$ is a discrete group of projective transformations preserving $\Omega$ acting freely and properly discontinuously on $\Omega$. The surface $S$ is $\emph{properly convex}$ if $\Omega$ is bounded in some affine space.
\end{definition}

%There is a well-known equivalent way of defining a convex $\RP$-surface in terms of the existence of a pair of maps with special properties. This is the following:
%\begin{theorem}[Development Theorem]Let $S$ be an $\RP$-surface, then the following are equivalent:\begin{itemize} \item[(1)] $S$ is convex \item[(2)] There exists a pair $(\dev, h)$, where $\dev:\widetilde{S}\to\RP$ is a diffeomorphism onto a convex domain in $\RP$ called the developing map and $h:\pi_1(S)\to\SL(3,\R)$ is a group homomorphism called the holonomy homomorphism, such that the following diagram commutes:\begin{equation}\begin{tikzcd}\widetilde{S} \arrow[r, "\mathrm{dev}"] \arrow[d, "\gamma"'] & \mathbb{R}\mathbb{P}^2 \arrow[d, "h(\gamma)"] \\\widetilde{S} \arrow[r, "\mathrm{dev}"]                         & \mathbb{R}\mathbb{P}^2         \end{tikzcd}\end{equation}Moreover, if $(\widetilde{dev},\widetilde{h})$ is another such pair, then $\exists g\in\SL(3,\R)$ such that:\begin{equation*}   \widetilde{\dev}=g\circ\dev, \qquad \widetilde{h}(\gamma)=g\circ h(\gamma)\circ g^{-1}, \ \forall\gamma\in\pi_1(S) \ . \end{equation*}\end{itemize}\end{theorem}
If $S$ is convex, then its universal cover $\widetilde{S}$ can be identified with a convex domain $\Omega\subset\RP$ via the \emph{developing map} and the discrete subgroup $\Gamma$ can be identified with $\pi_1(S)$ via the \emph{holonomy homomorphism}. From this point on we will focus only on the case in which the surface is closed, orientable and of genus $g$, hence it will be denoted with $\Sg$.\begin{definition}\label{defproperlyconvexstructure}A \emph{(properly) convex} $\RP$-\emph{structure} on $\Sg$ is a pair $(\phi,M)$, where $\phi:\Sg\to M$ is a diffeomorphism (called the \emph{marking}) and $M\cong \Omega/_{\Gamma}$ is a (properly) convex $\RP$-surface. 
\end{definition}One can define an equivalence relation on such pairs, namely $(\phi_1,M_1)\sim (\phi_2,M_2)$ if and only if there exists a projective isomorphism $\Psi:M_1\to M_2$ such that the new marking $\Psi\circ \phi_1$ is isotopic to $\phi_2$. The \emph{deformation space of (properly) convex $\RP$-structures} is defined as:
\begin{align}
    &\defg:=\{(\phi,M) \ \text{convex} \ \RP-\text{structure on} \ \Sg\}_{/_\sim} \\ &\defgp:=\{(\phi,M) \ \text{properly convex} \ \RP-\text{structure on} \ \Sg\}_{/_\sim} \ .
\end{align}
%\begin{remark}The reader familiar with Teichm\"uller theory may have noticed the remarkable similarity of the construction of this space with $\Teich(\Sg)$, when $g\ge 2$. In fact, as stated before, if $\Sg\cong\Hyp/_{\Gamma}$ with $\Gamma\subset\PSL(2,\R)$, using the Klein-Beltrami model for hyperbolic geometry and noting that $\PSL(2,\R)\cong\Iso^+(\Hyp)$ and $\PSL(3,\R)\cong\Aut(\RP)$, we can identify $\Gamma$ with a discrete subgroup of $\PSL(3,\R)$, contained in $\Proj(\Omega)$, acting freely and properly discontinuously on $\Omega\equiv\Hyp$. Thus, giving an inclusion $\Teich(\Sg)\subset\defg$ at least as sets.\footnote{We have implicitly used the existence of a unique (up to conjugation) irreducible representation $\imath:\PSL(2,\R)\hookrightarrow\PSL(3,\R)$.}\end{remark}
%The behavior of this space depends highly on the genus of the surface and, as one can imagine, there are notable differences between the flat case (genus one) and the hyperbolic one ($g\ge 2$). The first result in this direction was given by Kuiper and Benzecr\'i (\cite{kuiper1953convex},\cite{benzecri1960varietes}) which we now recall.
In the case of a hyperbolic Riemann surface, it is well-known that $\defg\equiv\defgp$ (see \cite{kuiper1953convex} and \cite{benzecri1960varietes}).
%\begin{proposition}If $\Sg$ is a convex $\RP$-surface with $g\ge 2$, then it must be properly convex. Moreover, the boundary $\partial\Omega$ is always strictly convex and $C^1$, and it must be either and ellipse or a Jordan curve which is nowhere $C^2$. In particular, there is an identification $\defg\equiv\defgp$.\end{proposition}
In the case of the torus this is no longer true, for instance there are many convex $\RP$-structures which are not properly convex: affine and Euclidean ones. They can not be properly convex since the developing map identifies the universal cover of $T^2$ with a copy of $\mathbb{R}^2$ inside $\RP$, which is convex but not bounded, see \cite[\S 8.5]{goldman2018geometric}. In particular, there is a strict inclusion $\deft\subsetneq\defp$ and the natural space to study in this contest is $\deft$, given that $\defp$ is not even a Hausdorff topological space. 

%On the other hand, much can be said about the properly convex domain $\Omega\subset\RP$ which covers the torus.

%\begin{proposition}
%Let $\Omega/_{\Gamma}$ be a properly convex $\RP$-structure on $T^2$, then $\Omega$ is projectively equivalent to a triangle in $\R^3$ with vertices $\{(1,0,0); (0,1,0); (0,0,1)\}$.
%\end{proposition}
%A detailed proof of this result can be found in \cite{rungi2021pseudo}, though it has already been proven differently using the classification of flat hyperbolic affine spheres in $\R^3$ up to unimodular affine transformations (\cite{magid1990flat}).

\subsection{A natural isomorphism}
Here we briefly introduce the theory of hyperbolic affine spheres in $\R^3$, recalling also the most important result regarding their global geometry as hypersurfaces. Then, we will focus on the close relationship between these geometric objects and cubic holomorphic differentials defined on the surface. All this will lead us to a parameterization of properly convex $\RP$-structures in terms of complex structures and holomorphic cubic differentials on $T^2$. \\ \\  Let $f\!: M \to \R^{3}$ be an immersion and $\xi\!:M\to\R^{3}$ be a transverse vector field to $f(M)$. This means that for all $x\in M$ we have a splitting: $$T_{f(x)}\R^{3} = f_* T_xM + \R\xi_x \ . $$ 
Let $D$ be the standard flat connection on $\R^3$ and suppose the structure equations of the immersed surface are given by:
\begin{equation}\begin{aligned}\label{structurequations}
D_XY &= \nabla_XY + h(X,Y)\xi \\
D_X\xi &= -S(X)
\end{aligned}\end{equation}
where $\nabla$ is a torsion-free connection on $M$ called the \emph{Blaschke connection}, $h$ is a metric on $M$ called the \emph{Blaschke metric} and $S$ is and endomorphism of $TM$ called the \emph{affine shape operator}.
\begin{definition}
Let $M$ be an immersed hypersurface in $\R^{3}$ with structure equations given by (\ref{structurequations}). Then $M$ is called a \emph{hyperbolic affine sphere} if $S=-\Id_{TM}$. 
\end{definition}

%It is not hard to see that the function $\lambda$ has to be constant and hence, by rescaling, we can assume $\lambda\in\{-1,0,1\}$. We say that an affine sphere is \emph{hyperbolic} if $\lambda=-1$.

%The global geometry of hyperbolic affine spheres is quite complicated. Their properties were conjectured by Calabi (\cite{calabi1972complete}) and proved by Cheng-Yau (\cite{cheng1977regularity}, \cite{cheng1986complete}) and Calabi-Nirenberg (with clarifications by Gigena (\cite{gigena1981conjecture}) and Li (\cite{li1990calabi}, \cite{li1992calabi})). The most important result, which holds in arbitrary dimension, is the following:
\begin{theorem}[\cite{cheng1977regularity},\cite{cheng1986complete},\cite{gigena1981conjecture},\cite{li1990calabi}]\label{thmchengyau}
Given a convex, bounded domain $\Omega\subset\R^3$, there is a unique properly embedded hyperbolic affine sphere $M\subset\R^{3}$ with center at $0$ asymptotic to the boundary of the cone $\mathcal{C}(\Omega):=\{(tx,t) \ | \ x\in\Omega, t>0\}\subset\R^{3}$. For any immersed hyperbolic affine sphere $f: M\rightarrow\R^{3}$, properness of the immersion is equivalent to the completeness of the Blaschke metric, and any such $M$ is a properly embedded hypersurface asymptotic to the boundary of the cone given by the convex hull of $M$ and its center.
\end{theorem}
 \begin{remark}\label{rmk:correspondencehypaffinesphereandpropconvex}
Theorem \ref{thmchengyau} induces a correspondence between equivariant hyperbolic affine spheres in $\R^3$ and properly convex $\RP$-structures on a closed orientable surface $\Sg$ of genus $g\ge 1$.\end{remark}
Let $f:\!(M, h, \nabla)\hookrightarrow\R^3$ be an immersed affine sphere, namely an immersed hypersurface satisfying equations (\ref{structurequations}) with $S=-\Id_TM$. If $\widehat{\nabla}$ denotes the Levi-Civita connection of $h$, then $\nabla=\widehat\nabla +A$, where $A$ is a section of $T^*(M)\otimes\End(TM)$. In particular, for every $X\in\Gamma(TM)$ the quantity $A(X)$ is an endomorphism of $TM$.
\begin{proposition}[{\cite[Lemma 4.3, Lemma 4.4]{benoist2013cubic}}]\label{prop:pickform}
The section $A$ has the following properties:
\begin{itemize}
    \item[(1)] $A(X)Y=A(Y)X, \qquad\forall X,Y\in\Gamma(TM)$
    \item[(2)] The endomorphism $A(X)$ is trace-free and $h$-symmetric, $\forall X\in\Gamma(TM)$
    \item[(3)] $\mathrm{d}^{\widehat\nabla}A=0$, where $\mathrm{d}^{\widehat\nabla}A(X,Y)=\big(\widehat\nabla_XA\big)(Y)-\big(\widehat\nabla_YA\big)(X),\qquad X,Y\in\Gamma(TM)$
\end{itemize}
\end{proposition}

\begin{definition}\label{def:picktensor}
The \emph{Pick tensor} is the $(0,3)$-tensor defined by \begin{equation}\label{picktensorandpickform} C(X,Y,Z):=h(A(X)Y,Z), \quad\forall X,Y,Z\in\Gamma(TM)\end{equation}
\end{definition}

\begin{proposition}\label{picktensoraffinesphere}
If $f:\!(M, h, \nabla)\hookrightarrow\R^{3}$ is an immersed affine sphere, then the Pick tensor is totally symmetric, namely in index notation $C_{ijk}$ we have $$C_{ijk}=C_{\sigma(ijk)},\quad\forall\sigma\in\mathfrak S_3 \ .$$
%In particular, the following relation holds $$ \big(\nabla_Xh\big)\big(Y,Z\big)=-2 C(X,Y,Z) \ . $$
\end{proposition}

%The last two results hold in a more general context where $M$ is not necessarily an immersed affine sphere, see \cite[\S 4.2]{benoist2013cubic}. 

\begin{theorem}[{\cite[Lemma 4.8]{benoist2013cubic}}]\label{thm:picktensor} Let $M$ be a smooth orientable surface endowed with a metric $h$ and a compatible complex structure $J$. Suppose the tensor $C(X,Y,Z)=h(A(X)Y,Z)$ is totally symmetric, then $A(X)$ is trace-free and $\mathrm{d}^{\widehat\nabla}A=0, \ \forall X\in\Gamma(TM)$ if and only if $C=\Re(q)$, where $q$ is a cubic holomorphic differential on $(M,J)$. If this is the case, then $q=C(\cdot,\cdot,\cdot)+iC(J\cdot,J\cdot,J\cdot)$.
\end{theorem}
\begin{remark}
Thanks to Proposition \ref{picktensoraffinesphere} we notice that if $M$ is an immersed affine sphere in $\R^3$ its Pick tensor is totally symmetric, hence it can always be expressed as the real part of a cubic holomorphic differential on $(M,J)$.
\end{remark}
We conclude this section by giving a parametrization of $\mathcal{B}_{0}(T^{2})$ which will be useful for our discussion. First we recall that (one of the possible) definition of the Teichm\"uller space $\Teich(T^2)$ is the following:$$\almost:=\{J\in\End(\R^2) \ | \ J^2=-\mathds{1}, \ \rho(v,Jv)>0 \ \text{for some} \ v\in\R^2\setminus\{0\}\} \ , $$ where $\rho:=\dx\wedge\dy$ is the standard area form on $\R^2$ (see \cite[\S 10.2]{farb2011primer}). In particular, this set has a smooth manifold structure of real dimension two. Let $\pi:\cubic\to\Teich(T^2)$ be the holomorphic vector bundle of cubic differentials, i.e. a point in $\cubic$ is represented by a pair $(J,q)$ where $J\in\Teich(T^2), q\in H^0\big(T^2,K^{\otimes^3}\big)\cong\C$ and $\pi(J,q):=J$. Let $\cubiccomplement:=\cubic\setminus\big(\Teich(T^2)\times\{0\}\big)$ be the complement of the zero section, then:
\begin{corollary}[\cite{rungi2021pseudo}]\label{cor:cubictorusandproperlyconvex}
There exists a bijection between $\deft$ and $\cubiccomplement$.
\end{corollary}
\begin{proof}The above bijection follows from Theorem \ref{thmchengyau}. In fact, by Remark \ref{rmk:correspondencehypaffinesphereandpropconvex} to any $\big[\Omega/_{\Gamma}\big]\in\deft$ we have an associated equivariant hyperbolic affine sphere $M$ in $\R^3$ which is determined by its Blaschke metric and the Pick tensor (see Definition \ref{def:picktensor} and \cite{nomizu1994affine}). Hence, let $\chi:\deft\to Q^3_0(\Teich(T^2))$ be the map that associates to each $\big[\Omega/_{\Gamma}\big]$ the pair $(J,q)$ where $J$ is the complex structure induced by the Blaschke metric and $q=c\mathrm{d}z^3$ is a non-zero cubic holomorphic differential whose real part coincides with the Pick tensor of $M$. It can be shown (see \cite{changping1990some},\cite[\S 2.3]{rungi2021pseudo}) that for any such $(J,q)$ we can find a unique (up to unimodular affine transformations) hyperbolic affine sphere in $\R^{3}$ that is invariant under a subgroup $\Gamma < \SL(3,\R)$ isomorphic to $\pi_{1}(T^{2})$. This implies directly that the map $\chi$ is a bijection. 
\end{proof}
\subsection{A new family of metrics}\label{sec:3.1}
In the following we want to introduce the new metrics defined on the space $\cubic$ which can be restricted to $\deft$. To do this, local (actually global) coordinates must be introduced on the holomorphic bundle of cubic differentials.\\ \\ Notice that the tangent space of $\almost$ at a point $J$ can be identified with $$T_J\almost=\{\dot J\in\End(\R^2) \ | \ \dot J J+J\dot J=0\}$$ and it is equipped with a natural K\"ahler structure $$\langle \dot J,\dot J'\rangle_J:=\frac{1}{2}\tr(\dot J\dot J'),\quad \Omega_J(\dot J,\dot J'):=-\langle \dot J,J\dot J'\rangle,\quad \mathcal I_J(\dot J):=-J\dot J \ . $$ 
\begin{lemma}[{\cite[Lemma 4.3.2]{trautwein2018infinite}}]\label{lem:almostkahler}
Let $\Hyp$ be the hyperbolic plane with complex coordinate $z=x+iy$ and with K\"ahler structure $$g_{\Hyp}=\frac{\mathrm{d}x^2+\mathrm{d}y^2}{y^2} \qquad\qquad \omega_{\Hyp}=-\frac{\mathrm{d}x\wedge\mathrm{d}y}{y^2} \ . $$ Then, there exists a unique K\"ahler isometry $j:\Hyp\to\almost$ such that $j(i)=\begin{pmatrix}0  & -1 \\ 1 & 0
\end{pmatrix}$. It is given by the formula \begin{equation}\label{kahlerisometryalmost}
    j(x+iy):=\begin{pmatrix}
    \frac{x}{y} & -\frac{x^2+y^2}{y} \\ \frac 1{y} & -\frac x{y} 
    \end{pmatrix} \ . 
\end{equation}
\end{lemma}
\begin{remark}
The minus sign in front of the area form on $\Hyp$ shows up since we are considering the relation $\Omega(\cdot,\cdot)=\langle\cdot,\mathcal I\cdot\rangle$ on $\almost$.
\end{remark}

In particular, thanks to this last lemma and the isomorphism $\Teich(T^2)\cong\almost$, we can identify the Teichm\"uller space of the torus with $\Hyp$. %Whenever we are thinking of the Teichm\"uller space of the torus as $\Hyp$, we will denote the total space of $\cubic$ as $\cubichyp$. 
Since $\Teich(T^2)$ is contractible, we can identify $\cubic$ with $\Hyp\times\C$, where $\C$ is a copy of the fiber $\cubic_z$ over a point $z\in\Hyp$. Let $z=x+iy$ and $w=u+iv$ be (gobal) coordinates on $\Hyp$ and $\C$, respectively. There is an $\SL(2,\R)$-action defined on $\Hyp\times \C$ given by
%\begin{equation}\label{SL2Raction}
$$\begin{pmatrix}
a & b \\ c & d
\end{pmatrix}\cdot(z,w):=\bigg(\frac{az+b}{cz+d}, (cz+d)^3w\bigg),\qquad \text{with} \ (z,w)\in\Hyp\times\C,\quad ad-bc=1 \ . $$
%\end{equation}
There is also an $S^1$-action on the fiber \begin{equation}\label{S1action}\theta\cdot(z,w):=(z,e^{i\theta}w), \quad\theta\in\R\end{equation} which is simply the rotation of the point $w\in\C$ and it fixes the copy of the hyperbolic plane $\Hyp\times\{0\}$. Moreover, there is a metric on the fibre $\C$, namely the one induced by the norm 
\begin{equation}\label{normvertical}|w|_z^2=\Imm(z)^3|w|^2 \qquad \text{for} \ z\in\Hyp, w\in\cubic_z \ .\end{equation} \begin{remark}\label{matrixSL2}
Notice that the $\SL(2,\R)$ action on the first factor $\Hyp$ is the usual action by M\"obius transformations, hence for any $z\in\Hyp$ there exists $P\in\SL(2,\R)$ such that $P\cdot i=z$. The matrix $P$ is given by $$P=\frac{1}{\sqrt y}\begin{pmatrix} 1 & -x \\ 0 & y
\end{pmatrix} \ . $$ 
\end{remark}
%Pick $J\in\almost$ and let us define the space of $J$-complex symmetric tri-linear forms by \begin{align*}S_3(\R^2, J):&=\{\gamma:\R^2\otimes\R^2\otimes\R^2\longrightarrow\C \ | \ \gamma \ \text{ is symmetric and} \ (J,i)-\text{tri-linear}\} \\ &\cong\{\tau:\R^2\to\C \ | \ \text{for all} \ \alpha,\beta\in\R \ \text{and} \ v\in\R^2 \ \text{it holds} \ \tau(\alpha v+\beta Jv)=(\alpha+i\beta)^3\tau(v)\} \ . \end{align*}
%This space can be seen as the fiber of a complex line bundle $\mathcal{L}_3(\R^2)\to\almost$ endowed with a natural $\SL(2,\R)$-action $$P\cdot (J,\gamma):=(PJP^{-1}, (P^{-1})^*\gamma),\qquad \ \text{for} \ P\in\SL(2,\R) \ .$$
%\begin{lemma}[{\cite[Lemma 5.2.1]{trautwein2018infinite}}]\label{lem:Trautweinfibremap}
%Define the map $\varphi:\cubichyp\to\Hom(\R^2\otimes\R^2\otimes\R^2, \C)$ by\begin{align*}\varphi(z,w) : \ &\R^2\longrightarrow\C \\ & v\mapsto \bar w(v_1-\bar zv_2)^3\end{align*}and let $j:\Hyp\to\almost$ be the map in (\ref{kahlerisometryalmost}). Then, the following holds: \begin{itemize}
    %\item[(1)]$\varphi(z,w)\in S_3(\R^2,j(z)), \ \text{for all} \  (z,w)\in\cubichyp$.
    %\item[(2)]The fibre map $\varphi(z,\cdot) : \cubichyp_z\cong\C\to S_3(\R^2,j(z))$ is a complex anti-linear isometry for every $z\in\Hyp$.
    %\item[(3)] The bundle map $(j,\varphi): \cubichyp\to\mathcal{L}_3(\R^2)$ is a $\SL(2,\R)$-equivariant bijection.\end{itemize}\end{lemma}
Let $\{\frac{\partial}{\partial x},\frac \partial{\partial y},\frac\partial{\partial u},\frac\partial{\partial v}\}$ be a real basis of the tangent space of $\Hyp\times\C$, and let $f:[0,+\infty)\to(-\infty,0]$ be a smooth function such that:$$
    (i) \ f(0)=0, \qquad(ii) \ f'(t)<0, \quad\forall t\ge 0,\qquad (iii) \lim_{t\to+\infty}f(t)=-\infty \ .$$
Then we define the following symmetric bi-linear form
\begin{equation*}
    \g_{(z,w)}=\begin{pmatrix}
    \frac 1{y^2}\big(1-f+3(u^2+v^2)y^3f'\big) & 0 & 2f'vy^2 & -2f'uy^2 \\ 0 & \frac 1{y^2}\big(1-f+3(u^2+v^2)y^3f'\big) & 2f'uy^2 & 2f'vy^2 \\ 2f'vy^2 & 2f'uy^2 & \frac 4{3}f'y^3 & 0 \\ -2f'uy^2 & 2f'vy^2 &0 & \frac{4}{3}f'y^3
\end{pmatrix}\end{equation*}
and the $2$-form
\begin{align*}\ome_{(z,w)}=&\bigg(-1+f-3f'y^3(u^2+v^2)\bigg)\frac{\dx\wedge\dy}{y^2}-\frac 4{3}f'y^3\du\wedge\devu \\&-2y^2f'\bigg(u(\dx\wedge\du+\dy\wedge\devu)+v(\du\wedge\dy-\devu\wedge\dx)\bigg)\end{align*}
where the functions $f,f'$ are evaluated at $y^3(u^2+v^2) \ .$
Let $\i_{(i,w)}: T_{(i,w)}\big(\Hyp\times\C\big)\to T_{(i,w)}\big(\Hyp\times\C\big)$ be the following almost-complex structure \begin{equation*}
    \i_{(i,w)}=\begin{pmatrix}
    0 & -1 & 0 & 0 \\ 1 & 0 & 0 & 0 \\ 0 & 0 & 0 & -1 \\ 0 & 0 & 1 & 0
    \end{pmatrix} \ . 
\end{equation*}
written in the basis $\{\frac{\partial}{\partial x},\frac \partial{\partial y},\frac\partial{\partial u},\frac\partial{\partial v}\} \ .$ 
\begin{theorem}[\cite{rungi2021pseudo}]\label{teo:rungitamburelli}
The triple $(\g_f,\i,\ome_f)$ defines an $\SL(2,\R)$-invariant pseudo-K\"ahler structure on $\Hyp\times\C$, hence on $\cubic$. Moreover, it restricts to a $MCG(T^2)$-invariant pseudo-K\"ahler metric on the complement of $\Hyp\times\{0\}$, hence on $\deft$.
\end{theorem}

\begin{remark}
By exploiting the properties of the function $f$, it is clear that the metric $\g_f$ is positive definite on $\Hyp\times\{0\}$ and negative-definite on the fibre $\{0\}\times\C$, hence it is of neutral signature $(2,2)$.
\end{remark}

\section{Global Darboux coordinates for complete Lagrangian fibrations}\label{sec:completelagrangianfibrations}
We initially recall the basic notions of symplectic geometry and integrable systems that we will need later on (see for example \cite[\S 18]{da2008lectures}). Then, we begin the study of complete Lagrangian fibrations and their connection with completely integrable Hamiltonian systems. \begin{definition}
A \emph{Hamiltonian system} is a triple $(M,\omega,H)$, where $(M,\omega)$ is a symplectic manifold and $H\in C^\infty(M,\R)$ is a function, called the \emph{Hamiltonian} function.
\end{definition}If $(M,\omega)$ is a symplectic manifold and $f\in C^\infty(M,\R)$, then the \emph{Hamiltonian vector field} $\mathbb X_f\in\Gamma(TM)$ associated with $f$ is defined by the following property \begin{equation}\label{hamiltonianfield}\omega(\mathbb X_f,Y)=\mathrm df(Y), \quad\forall Y\in\Gamma(TM)\ .\end{equation}
\begin{definition}
Let $(M,\omega,H)$ be a Hamiltonian system. A function $f\in C^\infty(M,\R)$ is called an \emph{integral of motion} if $$\omega(\mathbb X_f,\mathbb X_H)=0 \ .$$In other words, any integral of motion $f$ is constant along the integral curves of $\mathbb X_H$.
\end{definition}
\begin{definition}
A Hamiltonian system $(M,\omega,H)$ is \emph{completely integrable} if it possesses $n=\frac{1}{2}\dim(M)$ integral of motions $f_1=H,f_2,\dots,f_n$ such that\begin{itemize}
    \item[(1)] The differentials $(\mathrm d f_1)_p,\dots,(\mathrm df_n)_p$ are linearly independent for each $p\in M$. \item[(2)] They are pairwise in involution, i.e. $\omega(\mathbb X_{f_i},\mathbb X_{f_j})=0$ for each $i,j=1,\dots,n$.
\end{itemize}
\end{definition}The first condition in the previous definition is called \emph{independence} and the second one is called \emph{involutivity}. Notice that one of the integral of motion can always be taken to be the Hamiltonian function of the system. Furthermore, at each $p\in M$, the Hamiltonian vector fields associated with the integral of motions span an isotropic subspace of $T_pM$.\newline One of the most relevant results of this theory is the following \begin{theorem}[Arnold-Liouville, \cite{arnol2013mathematical}]\label{teo:arnlod}
Let $(M,\omega,H)$ be a completely integrable Hamiltonian system of dimension $2n$ and with integral of motions $f_1=H,f_2,\dots,f_n$. Let $c\in\R^n$ be a regular value of the map $f=(f_1,\dots,f_n):M\to\R^n$. Then, \begin{itemize}
    \item[(1)] The level set $f^{-1}(c)$ is a Lagrangian submanifold of $M$.
    \item[(2)] If the Hamiltonian vector fields $\mathbb X_{f_1},\dots\mathbb X_{f_n}$ are complete on the level set $f^{-1}(c)$, then each connected component of $f^{-1}(c)$ is diffeomorphic to $\R^k\times T^{n-k}$, for some $0\le k\le n$. Moreover, that component has coordinates $\theta_1,\dots,\theta_n$ called angle coordinates, in which the flows of $\mathbb X_{f_1},\dots\mathbb X_{f_n}$ are linear.
    \item[(3)]There are coordinates $\psi_1,\dots,\psi_n$, called action coordinates such that the manifold $(M,\omega)$ is symplectomorphic to $(\R^{n+k}\times T^{n-k},\omega_0)$, where $\omega_0=\displaystyle\sum_{i=1}^n\theta_i\wedge\psi_i$.
\end{itemize}
\end{theorem}\begin{remark}\label{rem:lagrangianfibration} From a geometric point of view, regular level sets $f^{-1}(c)$ being lagrangian submanifolds implies that, in a neighborhood of a regular value, the map $f:M\to\R^n$ is a \emph{lagrangian fibration}, i.e. it is locally trivial and its fibers are lagrangian submanifolds. For more details see \cite{da2008lectures}. \end{remark}On the other hand, one of the main issue of this result is that the action coordinates $\psi_1,\dots,\psi_n$ are, in general, not the given integral of motions, since $\theta_1,f_1,\dots,\theta_n,f_n$ may not form a global Darboux chart for $\omega.$ As we will see in Section \ref{sec:4.2}, this problem can be solved for a particular class of Hamiltonian systems, into which $\deft$ falls.
\subsection{Complete Lagrangian fibrations}\label{sec:4.3}
%This section and the next one are the most technical parts of the paper. All the theory and the results developed from now on are needed to the proof of Theorem \ref{teo:B}. As explained in Section \ref{sec:1.3}, the main issue in our case comes from the fact that each fiber of $H:(\deft,\ome_f)\to B$ is neither compact nor simply-connected.\\ \\
Throughout this section a Lagrangian fibration is a triple $(\pi, M,B)$, where $(M,\omega)$ is a symplectic manifold, $B$ is an open subset of $\R^n$ contained in the set of regular values of $\pi$, the map $\pi:(M,\omega)\to B$ is a smooth surjective submersion and for each $b\in B$ the sumbmanifold $\pi^{-1}(b)$ is Lagrangian in $(M,\omega)$. \\ \\ Let $\pi:(M,\omega)\to B$ be a Lagrangian fibration and let $\alpha:B\to T^*B$ be a $1$-form. Define a vector field $X_{\pi^*\alpha}\in\Gamma(TM)$ by setting \begin{equation}\label{vecfield}
    \omega(X_{\pi^*\alpha},\cdot)=\pi^*\alpha \ .
\end{equation}
\begin{proposition}\label{prop:vecfieldvertical}
For all $\alpha,\beta\in\Gamma(T^*B)$ and for all $f\in C^\infty(B)$ we have:\begin{itemize}
    \item[(i)] $X_{\pi^*(\alpha+\beta)}=X_{\pi^*\alpha}+X_{\pi^*\beta}$
    \item[(ii)] $X_{\pi^*(f\alpha)}=(\pi^*f)X_{\pi^*\alpha}$
    \item[(iii)] $X_{\pi^*\alpha}\in\Ker\pi_*$
    \item[(iv)] $\big[X_{\pi^*\alpha},X_{\pi^*\beta}\big]=0$
\end{itemize}
\end{proposition}
\begin{proof}
Properties (i) and (ii) follow directly form the defining equation (\ref{vecfield}). Let $q^1,\dots,q^n$ be local coordinates on $V\subset B$ such that $\alpha=\displaystyle\sum_{i=1}^n\alpha_i\mathrm dq^i$ for some functions $\alpha_i:V\to\R, i=1,\dots,n$. Then, by properties (i) and (ii) $$X_{\pi^*\alpha}=\sum_{i=1}^n(\pi^*\alpha_i)X_{\pi^*\mathrm dq^i} \ .$$Since condition (iii) is pointwise, it suffices to prove that for all functions $f\in C^\infty(V)$ we get $X_{\pi^*\mathrm df}\in\Ker\pi_*$. Let $f$ be such a function and $Y\in\Ker\pi_*$. Since $\pi^*f$ is constant along the fibre of $\pi:M\to B$, it follows $Y(\pi^*f)=0$. On the other hand, $$0=Y(\pi^*f)=\big(\pi^*(\mathrm df)\big)(Y)=\omega(X_{\pi^*\mathrm df},Y)$$ by definition of $X_{\pi^*\mathrm df}$. Since the last equality holds for all vertical fields $Y$, we have $X_{\pi^*\mathrm df}\in(\Ker\pi_*)^{\perp^\omega}$. The map $\pi$ defines a Lagrangian fibration, which implies that $(\Ker\pi_*)^{\perp^\omega}=\Ker\pi_*$ and property (iii) follows.\newline For the last property, let $\alpha,\beta\in\Gamma(T^*B)$ and locally write $$\alpha=\sum_{i=1}^n\alpha_i\mathrm dq^i,\qquad\beta=\sum_{j=1}^n\beta_i\mathrm dq^i$$for smooth functions $\alpha_i,\beta_j$. Then \begin{align*}
    [X_{\pi^*\alpha},X_{\pi^*\beta}]&=\sum_{i,j=1}^n[(\pi^*\alpha_i)X_{\pi^*\mathrm dq^i},(\pi^*\beta_j)X_{\pi^*\mathrm dq^j}]=\sum_{i,j=1}^n\bigg((\pi^*\alpha_i)(\pi^*\beta_j)[X_{\pi^*\mathrm dq^i},X_{\pi^*\mathrm dq^j}] \\ &+(\pi^*\alpha_i)\big(X_{\pi^*\mathrm dq^i}(\pi^*\beta_j)\big)X_{\pi^*\mathrm dq^j}-(\pi^*\beta_j)\big(X_{\pi^*\mathrm dq^j}(\pi^*\alpha_i)\big)X_{\pi^*\mathrm dq^i}\bigg) \\ &=\sum_{i,j=1}^n(\pi^*\alpha_i)(\pi^*\beta_j)[X_{\pi^*\mathrm dq^i},X_{\pi^*\mathrm dq^j}]
\end{align*}where the first equality follows from properties (i)-(ii) and the the third one from property (iii). Thus, it suffices to show that for any $f,g\in C^\infty(B)$ one has $[X_{\pi^*\mathrm df},X_{\pi^*\mathrm dg}]=0$. Notice that the homomorphism \begin{align*}
    C^\infty(M)&\to\Gamma(TM) \\& f\mapsto X_{\mathrm df}
\end{align*}is a Lie algebra homomorphism with respect to the Poisson bracket $\{\cdot,\cdot\}_\omega$ and the Lie bracket $[\cdot,\cdot]$, where $\{f,g\}_\omega:=\omega(X_{\mathrm df},X_{\mathrm dg})$. In particular, for each $f,g\in C^\infty(B)$ $$[X_{\pi^*\mathrm df},X_{\pi^*\mathrm dg}]=X_{\mathrm d\{\pi^*f,\pi^*g\}_\omega}=0 \ .$$ The second equality follows from the fact that $\pi:(M,\{\cdot,\cdot\}_\omega)\to(B,0)$ is a Poisson morphism (see \cite{vaisman1994lectures}).
\end{proof}
\begin{definition}\label{def:completefibration}
A Lagrangian fibration $\pi:(M,\omega)\to B$ is \emph{complete} if for each compactly supported $1$-form $\alpha$ on $B$, the vector field $X_{\pi^*\alpha}$ defined by (\ref{vecfield}) is complete.
\end{definition}
Recall that a Lagrangian fibration is naturally associated with a completely integrable Hamiltonian system (see Remark \ref{rem:lagrangianfibration}). The next Proposition explains why the previous hypothesis of completeness on a Lagrangian fibration is on the one hand interesting from the point of view of geometry and on the other not too restrictive.
\begin{proposition}\label{prop:completefibration}
Let $\pi:(M,\omega)\to B$ be a Lagrangian fibration associated with a completely integrable Hamiltonian system $(M,H,\omega)$ with integral of motions given by $f_1=H,f_2,\dots,f_n$. If the Hamiltonian vector fields $\mathbb X_{f_1},\dots,\mathbb X_{f_n}$ are complete on $\pi^{-1}(b)$ for each $b\in B$, then the Lagrangian fibration $\pi:(M,\omega)\to B$ is complete.
\end{proposition}
\begin{proof}
Since the Hamiltonian vector fields are vertical and linearly independent, they point-wise generate the tangent space to each fibre $\pi^{-1}(b)$. Moreover, there exist $1$-forms $\alpha_i:B\to T^*B$ such that $$\omega(\mathbb X_{f_i},\cdot)=\pi^*\alpha_i,\qquad i=1,\dots,n$$ where $\omega(\mathbb X_{f_i},\cdot)=\mathrm df_i$ by (\ref{hamiltonianfield}). The $1$-forms $\alpha_i$ are point-wise linearly independent on $B$, indeed if $$a_1\alpha_1+\cdots+a_n\alpha_n=0,\ \text{for some} \ a_1,\dots,a_n\in C^\infty(B)$$ then the above sum is still equal to zero after taking the pullback via $\pi$. By using the defining property of the $\alpha_i$'s and the independence property of $f_1,\dots,f_n$ we get $(\pi^*a_1)(m)=\dots=(\pi^*a_n)(m)=0, \ \forall m\in M$, i.e. the functions $a_1,\dots,a_n$ are zero on the whole set $B$. Let $\alpha:U\subset B\to T^*U$ be a locally defined compactly supported $1$-form. We need to prove that the vector field $X_{\pi^*\alpha}$, defined by (\ref{vecfield}), is complete. By the above argument, there exist $n$ functions $g_1,\dots,g_n$ on $U$ such that $$\alpha=\sum_{i=1}^ng_i\alpha_i \ .$$ Since $\alpha$ has compact support on $U$, the functions $g_i$ have compact support on the same set as well. In particular, they are bounded on $U$. By properties (i) an (ii) of Proposition \ref{prop:vecfieldvertical} it follows that $$X_{\pi^*\alpha}=\sum_{i=1}^n(\pi^*g_i)X_{\pi^*\alpha_i}=\sum_{i=1}^n(\pi^*g_i)\mathbb X_{f_i} \ .$$ The vector field $X_{\pi^*(g_j\alpha_j)}=(\pi^*g_j)\mathbb X_{f_j}$ is complete for all $j=1,\dots,n$ by an application of the so-called "Escape Lemma" (see \cite{lee2013smooth}, Lemma 9.19), indeed the function $\pi^*g_j$ is bounded and $\mathbb X_{f_j}$ is complete. Moreover, since $[X_{\pi^*(g_i\alpha_i)},X_{\pi^*(g_j\alpha_j)}]=0$ for all $i,j=1,\dots,n$ by property (iv) of Proposition \ref{prop:vecfieldvertical}, it follows that $X_{\pi^*(g_i\alpha_i)}+X_{\pi^*(g_j\alpha_j)}$ defines a new complete vector field as it is the sum of two commuting complete vector fields. Applying in an iterative way the previous observation, we deduce that $X_{\pi^{*}\alpha}$ is complete, as well. %$$\widetilde{X}:=\displaystyle\sum_{i=1}^{n-1}(\pi^*g_i)X_{\pi^*\alpha_i}=\sum_{i=1}^{n-1}(\pi^*g_i)\mathbb X_{f_i}$$is complete. In the end, since $[\widetilde{X},X_{\pi^*(g_n\alpha_n)}]=0$ for the same reason as before, the vector field $$X_{\pi^*\alpha}=\widetilde X +X_{\pi^*(g_n\alpha_n)}$$ can be written as the sum of two commuting complete vector fields. Its flow is given by the composition of the flows associated with $\widetilde{X}$ and $X_{\pi^*(g_n\alpha_n)}$, thus it is complete.
\end{proof}
From now on, all Lagrangian fibrations will be complete.\newline Let $\alpha:U\to T^*U$ be a compactly supported locally defined $1$-form on the base and let $\phi_\alpha^t:\pi^{-1}(U)\to\pi^{-1}(U)$ be the flow of the vector field $X_{\pi^*\alpha}$ defined for all $t\in\R$. Since $X_{\pi^*\alpha}$ is vertical, its flow $\phi_\alpha^t$ lies along the fibres of $\pi:(M,\omega)\to B$ for all $t\in\R$. Furthermore, for each $\alpha_b\in T^*B$ there exists a compactly supported locally defined $1$-form $\alpha:U\to T^*U$ such that $\alpha(b)=\alpha_b$ and the value of $X_{\pi^*\alpha}$ at a point $m\in M$ only depends on $\alpha_b$ and not on the choice of $\alpha$. Therefore, for each $\alpha_b\in T^*B$ there is a well-defined diffeomorphism $$\phi_{\alpha_b}^1:=\phi_\alpha^1{|_{\pi^{-1}(b)}}:\pi^{-1}(b)\to\pi^{-1}(b)$$ where $\alpha\in\Gamma(T^*U)$ is a compactly supported form such that $\alpha(b)=\alpha_b$. In particular, for each $b\in B$ the map \begin{equation}\begin{aligned}\label{actionfibres}
    T^*_bB&\to\Diff(\pi^{-1}(b)) \\&\alpha_b\mapsto\phi_{\alpha_b}^1
\end{aligned}\end{equation}is a Lie group homomorphism, where $T^*_bB$ has the structure of an abelian Lie group with respect to the sum of covectors. In other words, for each $b\in B$, the map in ($\ref{actionfibres}$) defines a transitive action of $T^*_bB$ on $\pi^{-1}(b)$. In general, this action is not free and for instance one can consider the associated isotropy group, namely $$\Lambda_b:=\{\alpha_b\in T^*_bB \ | \ \phi_{\alpha_b}^1(m)=m, \ \forall m\in\pi^{-1}(b)\}$$ known as the \emph{period lattice}. It can be proved that it is a discrete subgroup of $T^*_bB$ isomorphic to $\mathbb Z^k$, with $k=1,\dots,n$ (see \cite{duistermaat1980global} for the case $k=n$ or \cite{fiorani2002liouville} in general). \begin{remark}
In the case $M=T^*B$ and $\omega=\Omega_{\text{can}}$, the transitive action is simply given by the sum of covectors and $\Lambda_b=0$ for each $b\in B$.
\end{remark}\begin{definition}[\cite{vaisman1994lectures}]\label{def:periodnet}
The subset $$\Lambda:=\bigcup_{b\in B}\Lambda_b\subset T^*B$$ is called the \emph{period net} associated with the complete Lagrangian fibration $\pi:(M,\omega)\to B$.
\end{definition}
\begin{lemma}\label{lem:pullback}
Let $\pi:(M,\omega)\to B$ be a complete Lagrangian fibration and let $\alpha:U\to T^*U$ be a locally defined $1$-form. Then, \begin{equation}
    \big(\phi_\alpha^1\big)^*\omega-\omega=\pi^*\mathrm d\alpha
\end{equation}
\end{lemma}
\begin{proof}
    The proof relies on the following computation \begin{align*}
        \big(\phi_\alpha^1\big)^*\omega-\omega&=\int_0^1\frac{\mathrm d}{\mathrm dt}\big(\phi_\alpha^t\big)^*\omega\mathrm dt \\ &=\int_0^1\big(\phi_\alpha^t\big)^*\big(\mathcal L_{X_{\pi^*\alpha}}\omega\big)\mathrm dt \\ &=\int_0^1\big(\phi_\alpha^t\big)^*\mathrm d\big(\omega(X_{\pi^*\alpha},\cdot)\big)\mathrm dt \tag{Cartan's magic formula} \\ &=\int_0^1\big(\pi\circ\phi_\alpha^t\big)^*\mathrm d\alpha\mathrm d t \tag{Equation (\ref{vecfield})} \\ &=\int_0^1\pi^*\mathrm d\alpha\mathrm d t=\pi^*\mathrm d\alpha \ . \tag{$\pi\circ\phi_\alpha^t=\pi$, for all $t$} 
    \end{align*}
\end{proof}

\begin{theorem}[\cite{sepenotes}]\label{teo:smoothmanifold}
Let $\pi:(M,\omega)\to B$ be a complete Lagrangian fibration and let $\Lambda$ be the associated period net. Then, \begin{itemize}
    \item[(1)] $\Lambda$ is a closed Lagrangian submanifold of $T^*B$.
    \item[(2)] The quotient $T^*B/_{\Lambda}$ is a smooth manifold.
\end{itemize}
\end{theorem}
\begin{proof}Since $\pi:(M,\omega)\to B$ is a surjective submersion, then for each $b\in B$ there exists a local section $\sigma:U\to\pi^{-1}(U)$ defined in an open neighbourhood $U$ containing $b$ (\cite{giachetta1997new}, Proposition 1.2.4). Fix such a section and consider the map \begin{align}\label{psialpha}
    \psi_\sigma:T^*U&\to\pi^{-1}(U)\nonumber \\& \alpha\mapsto\phi^1_\alpha\big(\sigma\circ p(\alpha)\big)
\end{align}where $p:(T^*B,\Omega_{\text{can}})\to B$. We first want to prove that $\psi_\sigma$ is a local diffeomorphism. Since dim$T^*U$=dim$\pi^{-1}(U)$ it suffices to prove that $\Ker \mathrm d_\alpha\psi_\sigma=\{0\}$ for all $\alpha\in T^*U$. Fix an element $\alpha_0\in T^*U$ and notice that if $X\in T_{\alpha_0}T^*U$ is tangent to the fibres of $p$, then $\mathrm d_{\alpha_0}\psi_\sigma(X)=0$ if and only if $X=0$. Therefore, if $\mathrm d_{\alpha_0}\psi_\sigma(Y)=0$ and $Y\neq 0$, then $\mathrm d_{\alpha_0}p(Y)\neq 0$. Any such vector $Y\in T_{\alpha_0}T^*U$ is mapped to a non-zero vector $\widetilde Y\in T_{\psi_\sigma(\alpha_0)}\pi^{-1}(U)$ such that $\mathrm d\pi_{\psi_\sigma(\alpha_0)}(\widetilde Y)\neq 0$, since the action in (\ref{actionfibres}) preserves the fibre of $\pi:(M\,\omega)\to B$ and $\sigma$ is an immersion. This is not possible as the vector field $\widetilde Y$ is vertical with respect to $\pi$. Hence, $\psi_\sigma$ is a local diffeomorphism. Now let $b_0\in U$ and $\alpha_0\in\Lambda_{b_0}$. By definition $\psi_\sigma(\alpha_0)=(\sigma\circ p)(\alpha_0)$. The map $\psi_\sigma$ is a local diffeomorphism, hence there exists an inverse $\psi_\sigma^{-1}$ defined on an open neighbourhood $V\subset\pi^{-1}(U)$ of $(\sigma\circ p)(\alpha_0)$. By shrinking $U$ if needed, we may assume that $U=\pi(V)$. The composition $$\alpha_\sigma:=\psi_\sigma^{-1}\circ\sigma:U\to T^*U$$ is a locally defined $1$-form, since $p=\pi\circ\psi_\sigma$. In particular, for all $b\in U$ we get $$\sigma(b)=\psi_\sigma\circ\alpha_\sigma(b)=\phi^1_{\alpha_\sigma(b)}\big(\sigma(b)\big)$$ which means that for all $b\in U$, $\alpha_\sigma(b)\in\Lambda_b|_U$. Define $W:=\psi_\sigma^{-1}(V)$ and since $\psi_\sigma^{-1}$ is an open map, $W$ is an open neighbourhood (diffeomorphic to $V$) of $\alpha_\sigma(b)$. In the end, the above argument shows that $\alpha_\sigma(U)\subset W\cap\Lambda$. In order to show that $\Lambda$ is a smooth submanifold of $T^*B$ it suffices to prove that $W\cap\Lambda\subset\alpha_\sigma(U)$, since that would mean that $\Lambda$ is locally given by the graph of the $1$-form $\alpha_\sigma$. Let $\beta\in W\cap\Lambda$, then there exists $m\in V=\psi_\sigma(W)$ such that $$m=\psi_\sigma(\beta)=\phi^1_\beta\big(\sigma\circ p(\beta)\big) \ .$$On the other hand, $\beta\in\Lambda_{p(\beta)}$ implies that for all $\widetilde m\in\pi^{-1}(p(\beta))$, $\phi^1_\beta(\widetilde m)=\widetilde m$. Therefore, $$\phi^1_\beta(\sigma\circ p(\beta))=\sigma\circ p(\beta) \ .$$ Putting all together we get $$\psi_\sigma(\beta)=\sigma\circ p(\beta)$$ and applying $\psi^{-1}_\sigma$ to both sides of the equality $$\beta=\psi^{-1}_\sigma\circ\sigma\circ p(\beta)=\alpha_\sigma\circ p(\beta) \ .$$Thus proving that $\beta\in\alpha_\sigma(U)$. This completes the proof that $\Lambda$ is a smooth submanifold of $T^*B$. In order to show $\Lambda$ is also closed, let $\{\beta_n\}\subset\Lambda$ be a sequence converging to $\beta\in T^*B$. By taking a small enough neighbourhood $\widetilde W$ of $\beta$ in $T^*B$, it is possible to ensure that all but finitely many $\beta_n$ lie in $\widetilde W$ and that there exists a local section $\sigma:\widetilde U:=p(\widetilde W)\subset B\to M$. Again, for all but finitely many $n$, we have $$\psi_\sigma(\beta_n)=\sigma\circ p(\beta_n)$$ since $\beta_n\in\Lambda_{p(\beta_n)}$ for all $n\in \mathbb N$. By continuity of $\psi_\sigma$ the left hand side of the above equation converges to $\psi_\sigma(\beta)$ and by continuity of $\sigma\circ p$ the right hand side to $\sigma\circ p(\beta)$. Therefore, $\psi_\sigma(\beta)=\sigma\circ p(\beta)$ which means that $\beta\in\Lambda$. It only remains to show that $\Lambda$ is Lagrangian in $(T^*B,\Omega_{\text{can}})$. Notice that any locally defined section $\alpha:U\to\Lambda|_U$ of $p|_\Lambda:\Lambda\subset T^*B\to B$ is a closed $1$-form. In fact, for any such $\alpha$ we get $\phi_\alpha^1=\Id$, which implies $\big(\phi_\alpha^1\big)^*\omega=\omega$. By Lemma \ref{lem:pullback} it follows that $\pi^*\mathrm d\alpha=0$. Since $\pi$ is a submersion, we get $\mathrm d\alpha=0$ as required. In the end, the closed submanifold $\Lambda$ is locally given by the image of closed $1$-forms, hence it is Lagrangian in $(T^*B,\Omega_{\text{can}})$. The proof of claim (1) is completed. \newline The proof of (2) relies on the following standard result (see \cite{vaisman1994lectures}): if $N$ is a smooth manifold and $\mathcal R$ is an equivalence relation on $N$ whose graph in $N\times N$ is a closed submanifold, then the quotient $N/_{\mathcal R}$ is a smooth manifold. \newline In our case, two elements $\alpha,\beta\in T^*B$ are equivalent if and only if $\alpha-\beta\in\Lambda$. The proof that $$Q:=\big\{(\alpha,\beta)\in T^*B\times T^*B \ | \ \alpha-\beta\in\Lambda\big\}$$ is a closed submanifold of $T^*B\times T^*B$ can be done in the same way as before. Indeed by repeating the construction above it follows that $Q\cap(W_1\times W_2)=\alpha_{\sigma_1}(U)\times\beta_{\sigma_2}(U)$, where $\sigma_1,\sigma_2:U\subset B\to \pi^{-1}(U)\subset M$ are local sections of $\pi$ and $W_i:=\psi_{\sigma_i}^{-1}(V)$ are the corresponding open neighbourhood of $\alpha_{\sigma_1}(b)$ and $\beta_{\sigma_2}(b)$, for some $b\in U=\pi(V)$.
\end{proof}
\begin{corollary}\label{cor:tildepsi}
A choice of a local section $\sigma:U\subset B\to \pi^{-1}(U)\subset M$ induces a diffeomorphism $$\widetilde\psi_\sigma:T^*U/_{\Lambda}|_U\to\pi^{-1}(U)$$which commutes with the projections onto $U$.
\end{corollary}

\begin{remark}\label{rem:trivialization}
The diffeomorphism $\widetilde \psi_\sigma$ can be thought of as a local trivialization for the Lagrangian fibration $\pi:(M,\omega)\to B$. In particular, it sends the zero section of $T^*U\to U$ to the image of $\sigma$. The main issue of this construction is that a complete Lagrangian fibration $\pi:(M,\omega)\to B$ may not admit a globally defined section and, therefore, there is no natural choice of locally defined sections $\sigma:U\to\pi^{-1}(U)$ to construct the "trivialization" $\widetilde\psi_\sigma$.
\end{remark}
By Theorem \ref{teo:smoothmanifold}, the vertical action of $\Lambda$ (sum of covectors) on the fibres of $T^*B$ induced by a section $\alpha:U\to \Lambda|_U$ is by symplectomorphisms with respect to $\Omega_{\text{can}}$. In particular, this implies that the quotient space $T^*B/_\Lambda$ inherits a symplectic form $\widetilde\omega$ which makes the induced projection $$\widetilde p:\big(T^*B/_\Lambda,\widetilde\omega\big)\to B$$ a complete Lagrangian fibration.
\begin{definition}\label{def:referencesymplectic}
Given a complete Lagrangian fibration $\pi:(M,\omega)\to B$ with period net $\Lambda\subset T^*B$, the complete Lagrangian fibration given by $$\widetilde p:\big(T^*B/_\Lambda,\widetilde\omega\big)\to B$$ is called the \emph{symplectic reference} fibration associated to $\pi:(M,\omega)\to B$.
\end{definition}
\begin{remark}
Any symplectic reference Lagrangian fibration admits a globally defined Lagrangian section, obtained as the image of the zero section $0:B\hookrightarrow T^*B$ inside $T^*B/_\Lambda$. In fact, if $q:(T^*B,\Omega_{\text{can}})\to (T^*B/_\Lambda,\widetilde\omega)$ is the quotient projection such that $q^*\widetilde\omega=\Omega_{\text{can}}$ and $s:=q\circ 0$, then $$s^*\widetilde\omega=0^*(q^*\widetilde\omega)=0^*\Omega_{\text{can}}=0$$ hence $s:B\to (T^*B/_\Lambda,\widetilde\omega)$ is a globally defined Lagrangian section.
\end{remark}

\subsection{The existence of global Lagrangian sections}\label{sec:4.4}
In this last section we briefly explain the topological obstruction for a complete Lagrangian fibration to admit a global section and, more in particular, a global Lagrangian section. The main result shows that if $B$ is a contractible open subset of $\R^n$, then both the aforementioned topological obstructions vanish. \\ \\ Let $\pi:(M,\omega)\to B$ be a complete Lagrangian fibration as in the previous sections and let $U_i,U_j\subset B$ be open subsets such that $U_i\cap U_j\neq\emptyset$. Pick sections $\sigma_i:U_i\to\pi^{-1}(U_i), \sigma_j:U_j\to\pi^{-1}(U_j)$ and construct local trivializations $\widetilde\psi_{\sigma_i},\widetilde\psi_{\sigma_j}$ as in Corollary \ref{cor:tildepsi}. Consider the diffeomorphism $$\widetilde\psi_{\sigma_j}^{-1}\circ\widetilde\psi_{\sigma_i}:T^*\big(U_i\cap U_j\big)/_{\Lambda}|_{U_i\cap U_j}\to T^*\big(U_i\cap U_j\big)/_{\Lambda}|_{U_i\cap U_j}$$which leaves the projection onto $B$ invariant and it sends the zero section to $\widetilde\psi_{\sigma_j}^{-1}(\sigma_i)$ (see Remark \ref{rem:trivialization}). It can be proved (\cite{duistermaat1980global},\cite{dazord1987probleme}) that $\widetilde\psi_{\sigma_j}^{-1}(\sigma_i)$ is the unique section $s_{ji}$ of $T^*\big(U_i\cap U_j\big)/_{\Lambda}|_{U_i\cap U_j}\to U_i\cap U_j$ satisfying $\phi^1_{s_{ji}}(\sigma_j)=\sigma_i$. Fixing a good open cover $\mathcal U=\{U_i\}_{i\in I}$ in the sense of Leary, i.e. all subsets $U_i$ and all finite intersections of these subsets are contractible, the above construction yields locally defined smooth sections $s_{ji}$ for each pair $i,j$ whose respective open sets in $\mathcal U$ intersect non-trivially. By definition, the family $s_{ji}$ defines a \v{C}ech $1$-cocyle for the cohomology of $B$ with coefficients in the sheaf $C^\infty\big(T^*B/_{\Lambda}\big)$ of smooth sections of $T^*B/_{\Lambda}\to B$ and, therefore, a cohomology class $\eta\in H^1\big(B,C^\infty\big(T^*B/_{\Lambda}\big)\big)$. Let $\mathcal F_{\Lambda}$ be the sheaf of smooth sections of $p|_\Lambda:\Lambda\to B$. There is a short exact sequence of sheaves (\cite{duistermaat1980global},\cite{dazord1987probleme}) \begin{equation}\label{shortexact}0\to\mathcal F_\Lambda\to C^\infty(T^*B)\to C^\infty\big(T^*B/_{\Lambda}\big)\to 0\end{equation}where the first map is induced by the inclusion $\Lambda\hookrightarrow T^*B$ and $C^\infty(T^*B)$ is the sheaf of $1$-forms of $B$. It is a standard result that the sheaf $C^\infty(T^*B)$ is \emph{fine}, in particular it is acyclic since $B$ is a paracompact Hausdorff space (\cite{gunning2015lectures}). Then, for all $k\ge 1$ we have $$H^k\big(B,C^\infty(T^*B)\big)\cong\{0\} \ .$$ The long exact sequence in cohomology induced by the short exact sequence in (\ref{shortexact}) induces an isomorphism $$\Phi:H^1(B,C^\infty(T^*B/_\Lambda)\big)\overset{\cong}{\longrightarrow} H^2(B,\mathcal F_\Lambda)$$
\begin{theorem}[\cite{duistermaat1980global},\cite{dazord1987probleme}]\label{teo:globalsection}
The image $\Phi(\eta)=:c_\Lambda\in H^2(B,\mathcal F_\Lambda)$ is called the Chern class associated with the Lagrangian fibration $\pi:(M,\omega)\to B$ and $c_\Lambda=0$ if and only if there exists a globally defined section $\sigma:B\to M$.
\end{theorem}
\begin{remark}The topological (indeed smooth) structure of a complete Lagrangian fibration $\pi:(M,\omega)\to B$ is completely determined by its \emph{period net} $\Lambda$ and its Chern class $c_\Lambda\in H^2(B,\mathcal F_\Lambda)$. More precisely, two complete Lagrangian fibrations are fiber-wise diffeomorphic if and only if they have diffeomorphic period nets and equal (up to diffeomorphism relating the period nets) Chern classes.\end{remark} In light of the results of the previous section it makes sense to ask for a symplectic classification of complete Lagrangian fibrations. In particular, one might be interested in understanding when the diffeomorphism $\widetilde\psi_\sigma$ of Corollary \ref{cor:tildepsi} can be chosen to be a symplectomorphism between $\big(T^*U/_\Lambda|_U,\widetilde\omega\big)$ and $\big(\pi^{-1}(U),\omega\big)$. The first step in this direction is the existence of local Lagrangian sections. \begin{theorem}[\cite{fiorani2002liouville}]\label{teo:locallagrangiansection}
Let $\pi:(M,\omega)\to B$ be a complete Lagrangian fibration. Then, for each $b\in B$ there exists a neighborhood $U\subset B$ of $b$ and a local Lagrangian section $\sigma:U\to\pi^{-1}(U)$.
\end{theorem}
\begin{corollary}\label{cor:psisymplectomorphism}
The diffeomorphism $\widetilde\psi_\sigma$ is a symplectomorphism from $\big(T^*U/_\Lambda|_U,\widetilde\omega\big)$ to $\big(\pi^{-1}(U),\omega\big)$ if and only if the local section $\sigma$ is Lagrangian.
\end{corollary}
\begin{proof}
Let $\alpha\in\Gamma(T^*U)$ be a locally defined $1$-form on $B$ and let $q:T^*U\to T^*U/_\Lambda|_U$ be the restricted quotient map. Then, $q\circ\alpha:U\to T^*U/_\Lambda|_U$ is a local section of the symplectic reference Lagrangian fibration associated with $\pi:(M,\omega)\to B$ (see Definition \ref{def:referencesymplectic}). Applying Lemma \ref{lem:pullback} we get \begin{align*}
    (\phi_\alpha^1)^*\omega&=\omega+\pi^*\mathrm d\alpha \\ &= \omega+\pi^*\alpha^*\Omega_{\text{can}}\tag{$\mathrm d\alpha=\alpha^*\Omega_{\text{can}}$} \\&=\omega+\pi^*(q\circ\alpha)^*\widetilde\omega \ . \tag{$q^*\widetilde\omega=\Omega_{\text{can}}$}
\end{align*}
Applying $\sigma^*$ to both sides of the equation and using $\pi\circ\sigma=\text{Id}_U$ we have $$(\phi_\alpha^1\circ\sigma)^*\omega=\sigma^*\omega+(q\circ\alpha)^*\widetilde\omega \ .$$ Moreover, since by definition $\phi_\alpha^1\circ\sigma=\widetilde\psi_\sigma\circ q\circ\alpha$ and $\widetilde p\circ q\circ\alpha=\text{Id}_U$ (see (\ref{psialpha}) and Definition \ref{def:referencesymplectic}), the last equality can be written as $$(q\circ\alpha)^*\big((\widetilde\psi_\sigma)^*\omega+\widetilde p^*\sigma^*\omega-\widetilde\omega\big)=0 \ .$$
\underline{\emph{Claim}}: There exists a locally defined $2$-form $\beta$ on $U\subset B$ such that $(\widetilde\psi_\sigma)^*\omega-\widetilde\omega=\widetilde p^*\beta$.\newline Assuming the claim we can conclude the proof of the theorem. In fact, if such $\beta$ exists we get $$(q\circ\alpha)^*\big(\widetilde p^*\beta+\widetilde p^*\sigma^*\omega\big)=0 \ .$$ Using again that $\widetilde p\circ q\circ\alpha=\text{Id}_U$ we obtain $\beta=\sigma^*\omega$, hence $$(\widetilde\psi_\sigma)^*\omega-\widetilde\omega=\widetilde p^*\sigma^*\omega \ .$$At this point it is clear that $\sigma$ is Lagrangian (i.e. $\sigma^*\omega=0$) if and only if $\widetilde\psi_\sigma$ is a symplectomorphism. Finally, the proof of the claim above can be found in \cite[Proposition 2.3]{gross1998special}.
\end{proof}
Let $\pi:(M,\omega)\to B$ be a complete Lagrangian fibration and choose a good open cover $\mathcal U=\{U_i\}_{i\in I}$ of $B$ such that there exists a local Lagrangian section $\sigma_i:U_i\to\pi^{-1}(U)$ for each $i\in I$ (Theorem \ref{teo:locallagrangiansection}). Using Corollary \ref{cor:psisymplectomorphism} we can apply verbatim the construction made at the beginning of the section replacing "diffeomorphism" with "symplectomorphism". In particular, we get the existence of local Lagrangian sections $s_{ji}$ for $\widetilde p:\big(T^*(U_i\cap U_j)/_\Lambda|_{U_i\cap U_j},\widetilde\omega\big)\to U_i\cap U_j$. Let us denote this sheaf of Lagrangian sections by $\mathcal Z^1(T^*B/\Lambda)$. As before, the family $\{s_{ji}\}_{i,j\in I}$ defines a \v{C}ech cohomology class $\xi\in H^1(B,\mathcal Z^1(T^*B/_\Lambda))$, called the \emph{Lagrangian Chern class} associated with the complete Lagrangian fibration $\pi:(M,\omega)\to B$.
\begin{proposition}\label{prop:coveringspace}
The map $p|_\Lambda:\Lambda\to B$ is a covering space.
\end{proposition}
\begin{proof}
Notice that the smooth submanifold $\iota:\Lambda\hookrightarrow T^*B$ intersects $T^*_bB$, for each $b\in B$, at the period lattice $\Lambda_b\cong\mathbb Z^k, 1\le k\le n$ (see Definiton \ref{def:periodnet}). Hence, the fibre $(p|_\Lambda)^{-1}(b)\cong \Lambda\cap T^*_bB\cong\Lambda_b\cong\mathbb Z^k$ is discrete. Since $p:T^*B\to B$ is a vector bundle, for each $b\in B$ there exists an open neighborhood $U_b$ such that $p^{-1}(U_b)\cong U_b\times T^*_bB$. In particular, \begin{align*}
    (p|_\Lambda)^{-1}(U_b)&=\iota^{-1}\big(p^{-1}(U_b)\big) \\ &\cong p^{-1}(U_b)\cap\Lambda \\ &\cong U_b\times(T^*_bB\cap\Lambda) \\ &\cong U_b\times\mathbb Z^k \ .
    \end{align*}
\end{proof}
\begin{theorem}[\cite{duistermaat1980global},\cite{dazord1987probleme}]\label{teo:lagrangiansection}
Let $\pi:(M,\omega)\to B$ be a complete Lagrangian fibration with period net $\Lambda$ and vanishing Chern class $c_\Lambda=0$. Then, it admits a global Lagrangian section if and only if $\xi=0$.
\end{theorem}
\begin{corollary}\label{cor:lagrangianglobalsection}
Let $\pi:(M,\omega)\to B$ be a complete Lagrangian fibration over a contractible open connected subset $B$ in $\R^n$. Then, it admits a global Lagrangian section $\sigma:B\to (M,\omega)$.
\end{corollary}
\begin{proof}
Let $\Lambda$ be the period net of the complete Lagrangian fibration. There exists a short exact sequence of sheaves $$0\to\mathcal F_\Lambda\to\mathcal Z^1(T^*B)\to\mathcal Z^1(T^*B/_\Lambda)\to 0,$$where $\mathcal Z^1(T^*B)$ denotes the sheaf of closed $1$-forms on $B$ and $\mathcal F_\Lambda$ is the sheaf of sections of the covering $p|_\Lambda:\Lambda\to B$. The sheaf $\mathcal Z^1(T^*B)$ can be equivalently described as the sheaf of Lagrangian sections of $p:(T^*B,\Omega_{\text{can}})\to B$. The long exact sequence induced in cohomology gives \begin{align*}
    \dots&\to H^1(B,\mathcal Z^1(T^*B))\to H^1(B,\mathcal Z^1(T^*B/_\Lambda))\overset{\delta}{\to}H^2(B,\mathcal F_\Lambda)\to \\ &\to H^2(B,\mathcal Z^1(T^*B))\to H^2(B,\mathcal Z^1(T^*B/_\Lambda))\to\dots
\end{align*}
Using the following isomorphism \begin{equation*}
    H^k(B,\mathcal Z^1(T^*B))\cong H^{k+1}_{\text{dR}}(B,\R),\quad \text{if} \ k\ge 1 \tag{see \cite{gunning2015lectures} and \cite{bott1982differential}}
\end{equation*}and the hypothesis that $B$ is contractible we get \begin{equation}
    H^1(B,\mathcal Z^1(T^*B/_\Lambda))\cong H^2(B,\mathcal F_\Lambda) \ .
\end{equation}
On the other hand, the sheaf $\mathcal F_\Lambda$ is the sheaf of sections of a covering space over $B$ (Proposition \ref{prop:coveringspace}), hence it is locally constant on $B$\footnote{A sheaf of abelian groups $\mathcal F$ on $X$ is locally constant if for each $x\in X$ there exists a neighborhood $U$ containing $x$ such that $\mathcal F|_U$ is a constant sheaf on $U$.}. It is a standard result that over a smooth manifold $X$, locally constant sheaves of abelian groups $\mathcal F_G$ (also known as \emph{local systems}) correspond to representations $\rho:\pi(X,x_0)\to \Aut(G)$ (see \cite{dimca2004sheaves} Proposition 2.5.1). In our case $X=B$ is contractible, thus any representation as above is trivial and the local system $\mathcal F_\Lambda$ is actually isomorphic to the constant sheaf $\underline{\mathbb Z}^k$, for some $1\le k\le n$ (see Definition \ref{def:periodnet}). In the end, \begin{align*}H^1(B,\mathcal Z^1(T^*B/_\Lambda))&\cong H^2(B,\mathcal F_\Lambda) \\ &\cong H^2(B,\underline{\mathbb Z}^k) \tag{sheaf cohomology} \\ &\cong H^2(B,\mathbb Z^k) \tag{singular cohomology} \\ &\cong\bigg(H^2(B,\mathbb Z)\bigg)^k \cong\{0\}\end{align*}In particular, both the Chern class and the Lagrangian Chern class of the fibration vanish. As a consequence of Theorem \ref{teo:globalsection} and Theorem \ref{teo:lagrangiansection} we get the existence of a global Lagrangian section $\sigma:B\to (M,\omega)$.
\end{proof}

\section{\texorpdfstring{$\deft$}{deft} as a complete Hamiltonian integrable system}
\subsection{The Hamiltonian actions}
Before analyzing the two actions in detail and prove the first main Theorem, let us recall some basic definitions and a standard result regarding symplectic actions and moment maps.
\begin{definition}\label{def:momentmap}
Let $G$ be a Lie group, with Lie algebra $\mathfrak g$, acting on a symplectic manifold $(M,\omega)$ by symplectomorphisms. We say the action is \emph{Hamiltonian} if there exists a smooth function $\mu:M\to\mathfrak g^*$ satisfying the following properties: \begin{itemize}
    \item[(i)] The function $\mu$ is equivariant with respect to the $G$-action on $M$ and the co-adjoint action on $\mathfrak g^*$, namely \begin{equation}
        \mu_{g\cdot p}=\Ad^*(g)(\mu_p):=\mu_p\circ \Ad(g^{-1})\in\mathfrak g^* \ .
    \end{equation}
    \item[(ii)]Given $\xi\in\mathfrak g$, let $X_\xi$ be the vector field on $M$ generating the action of the $1$-parameter subgroup generated by $\xi$, i.e. $X_\xi=\frac{\mathrm d}{\mathrm dt}\text{exp}(t\xi)\cdot p |_{t=0}$. Then, for every $\xi\in\mathfrak g$ we have\begin{equation}
        \mathrm d\mu^{\xi}=\iota_{X_\xi}\omega=\omega(X_\xi,\cdot)
    \end{equation}where $\mu^\xi:M\to\R$ is the function $\mu^\xi(p):=\mu_p(\xi)$.
\end{itemize}A map $\mu$ satisfying the two properties above is called a \emph{moment map} for the Hamiltonian action.
\end{definition}
\begin{lemma}\label{lem:restrictedhamiltonian}
Let $(M,\omega)$ be a symplectic manifold endowed with a Hamiltonian $G$-action and moment map $\mu_G:M\to\mathfrak g^*$. If $H\le G$ is any closed subgroup, then the restricted $H$-action is Hamiltonian with moment map $\mu_H:M\to\mathfrak h^*$ given by $\mu_H:=|_{\mathfrak h}\circ\mu_G$, where $|_{\mathfrak h}:\mathfrak g^*\to\mathfrak h^*$ is the map which associates to each functional on $\mathfrak g$ its restriction on $\mathfrak h$.
\end{lemma}
From Corollary \ref{cor:cubictorusandproperlyconvex} and Lemma \ref{lem:almostkahler} the deformation space $\deft$ is diffeomorphic to $\Hyp\times\C^*$. Since the circular action defined in (\ref{S1action}) preserves $\Hyp\times\{0\}$, for $\theta\in\R$ let $\Psi_\theta$ denotes the induced $S^1$-action on $\deft$. Then, we have the following \begin{theorem}[\cite{rungi2021pseudo}]\label{teo:rungitamburelli2}
The $S^1$-action on $\deft$ is Hamiltonian with respect to $\ome_f$ and it satisfies $$
     \Psi_\theta^{*}\ome_f=\ome_f, \qquad \Psi_\theta^{*}\g_f=\g_f \ .$$
     Moreover, the Hamiltonian function is given by $H_1(z,w)=\frac{2}{3}f\big(\Imm(z)^3|w|^2\big)$.
\end{theorem}
Since the $\SL(2,\R)$-action defined in Section \ref{sec:3.1} preserves $\Hyp\times\{0\}$ as well, we get 
\begin{theorem}[\cite{rungi2021pseudo}]\label{teo:rungitamburelli3}
The $\SL(2,\R)$-action on $\deft$ is Hamiltonian with respect to $\ome_f$ and with moment map $\mu:\deft\to\Lsl(2,\R)^*$ given by
\begin{equation}
    \langle \mu(z,w),X \rangle=\bigg(1-f\big(\Imm(z)^3|w|^2\big)\bigg)\tr(j(z)X),\quad X\in\Lsl(2,\R)
\end{equation}where $j:\Hyp\to\almost$ is the map in (\ref{kahlerisometryalmost}).
\end{theorem}
%From now on we will focus on the case when $f(t)=-t$ for all $t\ge 0$. In particular, the Hamiltonian function of the $S^1$-action is given by $H_1(z,w)=-\frac{2}{3}y^3(u^2+v^2)$. \\ \\
Notice that inside $\SL(2,\R)$ there is the subgroup of diagonal matrices with determinant equal to one, namely
\begin{equation}
    \bigg\{\begin{pmatrix}\lambda & 0 \\ 0 & \frac{1}{\lambda} \end{pmatrix} \ \bigg| \ \lambda\in\R^*\bigg\}<\SL(2,\R) \ .
\end{equation}In particular, such a subgroup can be identified with a copy of $\R^*$ which still acts in a Hamiltonian fashion (Lemma \ref{lem:restrictedhamiltonian}) on the space $\deft$.
\begin{lemma}
Let $\R^*$ be a copy of the subgroup of diagonal matrices in $\SL(2,\R)$ and consider its restricted Hamiltonian action on $\deft$, then the Hamiltonian function is given by $$H_2(z,w)=2\frac{x}{y}\big(1-f(y^3|w|^2)\big) \ .$$
\end{lemma}
\begin{proof}
The Lie algebra of $\R^*$ can be identified with $$\mathfrak h:=\bigg\{\begin{pmatrix}\alpha & 0 \\ 0 & -\alpha 
\end{pmatrix} \ \bigg| \ \alpha\in\R\bigg\} \ .$$ By Lemma \ref{lem:restrictedhamiltonian}, if $\mu$ denotes the moment map for the $\SL(2,\R)$-action of Theorem \ref{teo:rungitamburelli3}, then the associated moment map for the restricted $\R^*$-action $\mu_\mathfrak h:\deft\to\mathfrak h^*$ is $$\mu^X_\mathfrak h(z,w) =\big(1-f(y^3|w|^2)\big)\tr(j(z)X),$$ where $X\in\mathfrak h$. Let $\xi:=\begin{pmatrix}1 & 0 \\ 0 & -1
\end{pmatrix}\in\mathfrak h$, then the Hamiltonian function $H_2:\deft\to\R$ is $H_2(z,w):=\mu^\xi_\mathfrak h(z,w)$, given that $\mathrm d\mu^\xi_\mathfrak h=\ome_f(V_\xi,\cdot)$, where $$ V_\xi=2\bigg(x\frac{\partial}{\partial x}+y\frac{\partial}{\partial y}\bigg)-3\bigg(u\frac{\partial}{\partial u}+v\frac{\partial}{\partial v}\bigg)$$ is the infinitesimal generator of the action. Finally, since $$\tr(j(z)\xi)=\tr\bigg(\begin{pmatrix}\frac{x}{y} & -\frac{x^2+y^2}{y} \\ \frac 1{y} & -\frac x{y}
\end{pmatrix}\cdot\begin{pmatrix} 1 & 0 \\ 0 & -1
\end{pmatrix}\bigg)=2\frac{x}{y}$$ we get $\displaystyle H_2(z,w)=2\frac{x}{y}\big(1-f(y^3|w|^2)\big)$. 
\end{proof}
\subsection{Global Darboux coordinates}\label{sec:4.2}
In this section we prove Theorem \ref{thmB}, namely the main result regarding the symplectic geometry of $\big(\deft,\ome_f\big)$.
\begin{proposition}\label{prop:deftintegrablesystem}
The Hamiltonian system $(\deft,\ome_f,H_1)$ is completely integrable. The integral of motions are given by $$H_1(z,w)=\frac{2}{3}f\big(y^3|w|^2\big),\quad H_2(z,w)=2\frac{x}{y}\big(1-f(y^3|w|^2)\big) \ .$$
\end{proposition}
\begin{proof}
Let $\Xuno,\Xdue$ be the Hamiltonian vector fields associated with $H_1,H_2$. An explicit expression is given by $$\Xuno=u\frac{\partial}{\partial v}-v\frac{\partial}{\partial u},\quad \Xdue=2\bigg(x\frac{\partial}{\partial x}+y\frac{\partial}{\partial y}\bigg)-3\bigg(u\frac{\partial}{\partial u}+v\frac{\partial}{\partial v}\bigg) \ .$$ It is clear that they are point-wise linearly independent on $\deft$, hence to end the proof we only need to show that they are involutive. The symplectic form is \begin{align*}\ome_f&=\bigg(-1+f-3f'y^3(u^2+v^2)\bigg)\frac{\dx\wedge\dy}{y^2}-\frac 4{3}f'y^3\du\wedge\devu \\&-2y^2f'\bigg(u(\dx\wedge\du+\dy\wedge\devu)+v(\du\wedge\dy-\devu\wedge\dx)\bigg) \ .\end{align*}Moreover, \begin{align*}
    \ome_f(\Xuno,\Xdue)=2&
    \bigg(ux\ome_f\bigg(\partialv,\partialx\bigg)+uy\ome_f\bigg(\partialv,\partialy\bigg)-vx\ome_f\bigg(\partialu,\partialx\bigg)- \\ &vy\ome_f\bigg(\partialu,\partialy\bigg)\bigg)-3\bigg(u^2\ome_f\bigg(\partialv,\partialu\bigg)-v^2\ome_f\bigg(\partialu,\partialv\bigg)\bigg) \ .
\end{align*}It is easy to see that the term at the right hand side of the last equality is equal to zero, hence we get the claim.
\end{proof}
Let $H:=(H_1,H_2):(\deft,\ome_f)\to B$ be the Lagrangian fibration associated with the above completely integrable Hamiltonian system (see Remark \ref{rem:lagrangianfibration}), where $B:=H\big(\deft\big)\subset\R^2$. Using the explicit expression of the integral of motions and the properties of the function $f$, it is clear that $B$ is homeomorphic to $U:=\big\{(u_1,u_2)\in\R^2 \ | \ u_1<0\big\}$, hence it is contractible. In particular, any $b=(b_1,b_2)\in B$ is a regular value for $H$ and each fiber $$H^{-1}(b)=\big\{(z,w)\in\deft \ | \ \frac{2}{3}f\big(y^3|w|^2\big)=b_1, \ 2\frac{x}{y}\big(1-f(y^3|w|^2)\big)=b_2\big\}$$is diffeomorphic to $\R\times S^1$.
\begin{remark}\label{rem:fiberofH}
The fact that each fiber is diffeomorphic to $\R\times S^1$ can be seen directly from part (2) of Theorem \ref{teo:arnlod}, since the vector fields $\Xuno,\Xdue$ are complete on $H^{-1}(b)$, for each $b\in B$. Indeed $\Xuno$ is the generator of the counter clock-wise rotation in the plane and the integral curve of $\Xdue$ passing through the point $(z,w)$ is $\gamma_{(z,w)}(t)=(e^{2t}z,e^{-3t}w)$, which is defined for all $t\in\R$.
\end{remark}
\begin{manualtheorem}A\label{teo:B}
Let $(s,\theta)\in\R\times S^1$ be the angle coordinates of $(\deft, H_1, \ome_f)$ given by the Arnold-Liouville theorem. Then, $\big\{\theta, H_1, s, H_2\big\}$ is a global Darboux frame for $\ome_f$.
\end{manualtheorem}
\begin{proof}
The Lagrangian fibration $H:(\deft,\ome_f)\to B$ is the one arising from the completely integrable Hamiltonian system $(\deft,\ome_f,H_1)$. Since the vector fields $\Xuno,\Xdue$ are complete on each fiber $H^{-1}(b)$, by Proposition \ref{prop:completefibration} the Lagrangian fibration $H:(\deft,\ome_f)\to B$ is complete (see Definition \ref{def:completefibration}). Moreover, the base $B$ is a contractible open subset of $\R^2$. Using Corollary \ref{cor:lagrangianglobalsection} we get the existence of a global Lagrangian section $\sigma:B\to\deft$. In particular, $\sigma(B)$ is a Lagrangian submanifold of $(\deft,\ome_f)$, $\sigma(b)\in H^{-1}(b)$ for each $b\in B$ and $H\circ\sigma=\text{Id}_B$. Let $b=(b_1,b_2)\in B$, then the vector fields $$\frac{\partial}{\partial H_i}=\mathrm d\sigma\bigg(\frac{\partial}{\partial b_i}\bigg), \quad i=1,2$$ are tangent to $\sigma(B)$ and they generate a Lagrangian subspace of $T_{\sigma(b)}\deft$. In fact, \begin{align*}(\ome_f)_{\sigma(b)}\bigg(\partialHuno,\partialHdue\bigg)&=(\ome_f)_{\sigma(b)}\bigg(\mathrm d\sigma\bigg(\partialbuno\bigg),\mathrm d\sigma\bigg(\partialbdue\bigg)\bigg) \\ &=(\sigma^*\ome_f)_b\bigg(\partialbuno,\partialbdue\bigg) \\ &=0 \ . \tag{The section $\sigma$ is Lagrangian}\end{align*} Let $(s,\theta)\in\R\times S^1$ be the angle coordinates given by the Arnlod-Liouville Theorem, then the vector fields $\displaystyle\frac{\partial}{\partial\theta},\frac{\partial}{\partial s}$ are point-wise tangent to the fiber of $H:(\deft,\ome_f)\to B$. In particular, they correspond to $\mathbb X_{H_1}$ and $\mathbb X_{H_2}$ respectively. Hence,  \begin{equation*}
    \ome_f\bigg(\partialtheta,\partials\bigg)=\ome_f\big(\Xuno,\Xdue\big)=0 \ . \tag{Involution} 
\end{equation*}
Let us denote the coordinate $s$ with $g_1$ and $\theta$ with $g_2$. In order to conclude the proof of the theorem, we need to show that \begin{equation}
    (\ome_f)_x\bigg(\partialgi,\partialHj\bigg)=\delta^i_j,\quad\forall x\in\deft \ .
\end{equation}Suppose first $x\in\sigma(B)$, hence $(x_1,x_2)=\sigma(b_1,b_2)$ for some $(b_1,b_2)\in B$. Then, \begin{align*}
    (\ome_f)_x\bigg(\partialgi,\partialHj\bigg)&=(\ome_f)_x\bigg(\mathbb X_{H_i},\mathrm d_b\sigma\bigg(\partialbj\bigg)\bigg) \\ &=\mathrm d_xH_i\bigg(\mathrm d_b\sigma\bigg(\partialbj\bigg)\bigg) \tag{\ref{hamiltonianfield}} \\ &=\mathrm d_b\big(H_i\circ\sigma\big)\bigg(\partialbj\bigg) \tag{Chain rule} \\ &=\delta^i_j \ . \tag{$H_i\circ\sigma=b_i$}
\end{align*}Now let $x$ be an arbitrary point of $\deft$ and let $\Psi_i^t$ be the Hamiltonian flow associated with $H_i$. Since the flow action on the fiber $H^{-1}(b)$ is transitive, we can always assume that $x=\Psi_i^t\big(\sigma(H(x))\big)$, where $b=H(x)$. In particular, we have that the vector field $\displaystyle\partialHj$ computed at $x=\Psi_i^t\big(\sigma(H(x))\big)$ is equal to $\displaystyle\mathrm d\Psi_i^t\bigg(\frac{\partial}{\partial H_j}\bigg)$, where now the vector field inside the differential of $\Psi_i^t$ is computed at $\sigma(H(x))$. Hence, \begin{align*}
    (\ome_f)_x\bigg(\partialgi,\partialHj\bigg)&=(\ome_f)_x\bigg(\mathbb X_{H_i},\mathrm d\Psi_i^t\bigg(\frac{\partial}{\partial H_j}\bigg)\bigg) \\ &=\bigg(\big((\Psi_i^t)^{-1}\big)^*\ome_f\bigg)_x\bigg(\mathbb X_{H_i}, \mathrm d\Psi_i^t\bigg(\frac{\partial}{\partial H_j}\bigg)\bigg) \tag{$\Psi_i^t$ preserves $\ome_f$} \\ &=(\ome_f)_{\sigma(H(x))}\bigg(\big(\mathrm d\Psi_i^t\big)^{-1}\bigg(\mathbb X_{H_i}\bigg), \partialHj\bigg) \\ &=(\ome_f)_{\sigma(H(x))}\bigg(\mathbb X_{H_i},\partialHj\bigg) \tag{$\Psi_i^t$ is the flow associated with $H_i$} \\ &=\delta_i^j \tag{$\sigma(H(x))\in\sigma(B)$}
\end{align*}
\end{proof}

\section{Curvature properties}\label{sec:3}
In this section we show that the copy of the hyperbolic plane $\Hyp\times\{0\}\subset\Hyp\times\C$ is the only embedded submanifold with scalar curvature equal to $1$, whenever $f(t)=-kt$ with $k>0$. Initially, this last estimate will allow us to compute the connected component of the identity of the isometry group and secondly, we will deduce the expression for an arbitrary isometry.

\subsection{The Ricci tensor and the scalar curvature}
A \emph{K\"ahler manifold} is a particular case of a pseudo-K\"ahler one, namely when the pseudo-Riemannian metric has index equal to zero. For this very reason it is natural to ask whether some properties of K\"ahler manifolds still holds in this more general setting. Now we briefly recall the definition of some curvature tensors defined on K\"ahler manifolds and we will explain how their formulae still hold in the pseudo-Riemannian setting as long as the pseudo-metric is of neutral signature. \\ \\ Let $(M,g,I)$ be a K\"ahler manifold of complex dimension $n$. The tensor $I$ can be extended by $\C$-linearity on the complexified tangent bundle $T_\C M:=TM\otimes_{\R}\C$. Since $I^2=-\mathds{1}$ there is an eigenbundle decomposition $T_\C M=T^{1,0}\oplus T^{0.1}M$, where $$T^{1,0}M:=\{X\in T_\C M \ | \ I(X)=i\cdot X\},\qquad T^{0,1}M:=\{X\in T_\C M \ | \ I(X)=-i\cdot X\} \ .$$
The bundle $T^{1,0}M$ is called the \emph{holomorphic tangent bundle} and $T^{0,1}M$ the \emph{anti-holomorphic tangent bundle} of $M$, in particular they are the conjugate of each other. If $(z_1,\dots,z_n)$ are local holomorphic coordinates on $M$, the $n$-dimensional complex vector space $T^{1,0}M$ is generated by $\{\frac{\partial}{\partial z_1},\dots\frac{\partial}{\partial z_n}\}$. Since $z_k=x_k+iy_k$ for each $k=1,\dots,n$ we have $$\frac{\partial}{\partial z_k}=\frac{1}{2}\bigg(\frac{\partial}{\partial x_k}-i\frac{\partial}{\partial y_k}\bigg),\qquad \frac{\partial}{\partial \bar z_k}=\frac{1}{2}\bigg(\frac{\partial}{\partial x_k}+i\frac{\partial}{\partial y_k}\bigg),\quad\forall k=1,\dots n \ .$$
If we denote with $g^\C$ the $\C$-linear extension of $g$ to $T_\C M$, then locally it can be written as $$g^\C:=\sum_{j,k}g^\C_{j\bar k}\big(\mathrm dz^j\otimes\mathrm d\bar z^k+\mathrm d\bar z^k\otimes\mathrm dz^j\big)$$ where $g^\C_{j\bar k}:=g^\C\big(\frac{\partial}{\partial z_j},\frac{\partial}{\partial\bar z_k}\big)=\frac{1}{4}\bigg(g\big(\frac{\partial}{\partial x_j},\frac{\partial}{\partial x_k}\big)+g\big(\frac{\partial}{\partial y_j},\frac{\partial}{\partial y_k}\big)\bigg)$, since the Hermitian condition implies that $g^\C_{jk}=g^\C_{\bar j\bar k}=0$ and the symmetry that $\overline{g^\C_{j\bar k}}=g^\C_{k\bar j}$ for each $j,k=1,\dots,n$.\\ \\ In the following, by abuse of notation, we will denote with $g$ the metric extended by $\C$-linearity on $T_\C M$. If $\nabla$ denotes the Levi-Civita connection of $g$, then the only non vanishing Christoffel symbols are $$\Gamma_{jk}^i:=g^{i\bar l}\frac{\partial g_{k\bar l}}{\partial z_j},\qquad \Gamma_{\bar j\bar k}^{\bar i}:=\overline{\Gamma_{jk}^i}$$ where $g^{j\bar k}$ denotes the inverse metric computed on $\frac{\partial}{\partial z_j},\frac{\partial}{\partial\bar z_k}$. The Riemann curvature tensor $R\in\Gamma\big(T_\C^*M\otimes T_\C M\otimes\End(T_\C M)\big)$ of $\nabla$ is given by $${{R_i}^j}_{k\bar l}=-\frac{\Gamma_{ki}^j}{\partial\bar z_l},\qquad R_{\alpha\bar\beta\gamma\bar\delta}={{R_\alpha}^j}_{\gamma\bar\delta}g_{j\bar\beta} \ . $$As a consequence of the Bianchi identity, the Riemann tensor enjoys the following symmetries $$R_{i\bar j k\bar l}=R_{i\bar l k\bar j}=R_{k\bar j i\bar l}=R_{k\bar l i\bar j} \ . $$ Finally, the Ricci tensor and the scalar curvature are given, respectively, by: $$R_{i\bar j}=g^{k\bar l}R_{i\bar jk\bar l}, \qquad \text{scal}(g)=g^{i\bar j}R_{i\bar j} \ . $$
\begin{remark}
All the properties listed so far hold in the case of pseudo-K\"ahler manifolds, indeed they are only a consequence of the fact that the metric is non-degenerate and that $\nabla g=\nabla I=0$ (see \cite{zheng2001complex}).
\end{remark}
\begin{lemma}\label{lem:scalarepseudo}
Let $(M,g,I)$ be a pseudo-K\"ahler manifold of complex dimension $n$ and of neutral signature, then $$R_{i\bar j}=-\frac{\partial^2}{\partial z_i\partial\bar z_j}\log\big(\det(g)\big)$$
\end{lemma}
\begin{proof}
First, notice that $\log\big(\det(g)\big)$ is well-defined since the pseudo-metric $g$ is of neutral signature, hence $\det(g)>0$. Then, by using the formulae above, we get \begin{align*}
    R_{i\bar j}&=g^{k\bar l}R_{k\bar l i\bar j} \tag{symmetry of $R_{i\bar j k\bar l}$} \\ &={{R_p}^p}_{i\bar j} \tag{$R_{k\bar l i\bar j}=g_{p\bar l}{{R_k}^p}_{i\bar j}$} \\ &=-\frac{\partial\Gamma_{ip}^p}{\partial\bar z_j} \\ &=-\frac{\partial}{\partial \bar z_j}\big(g^{p\bar q}\frac{\partial g_{p\bar q}}{\partial z_i}\big) \\ &=-\frac{\partial^2}{\partial z_i\partial\bar z_j}\log\big(\det(g)\big) \ . \tag{Jacobi's formula}
\end{align*} 
\end{proof}
Before computing the Ricci tensor and the scalar curvature of the new metrics, it should be noted that it is sufficient to do the computation at points $(i,u)\in\Hyp\times\C$. In fact, the $\SL(2,\R)$-action introduced in Section \ref{sec:3.1} allows us to reduce to the points $(i,w)$ and the natural $S^1$-action on $\C$, presented in (\ref{S1action}), to the points $(i,u)$, since both actions are isometric (Theorem \ref{teo:rungitamburelli} and Theorem \ref{teo:rungitamburelli2}). Furthermore, we need to write $\det(\g_f)$ and the inverse of the metric $\g_f$, extended by $\C$-linearity on $T_\C(\Hyp\times\C)$, in terms of the coordinates $(z,w)$. We have: \begin{align*}
    &(\g_f^{z\bar z})_{(i,u)}=\frac{1}{2(1-f)}, \quad (\g_f^{w\bar w})_{(i,u)}=\frac{3(1-f+3f'u^2)}{8f'(1-f)} \\ &(\g_f^{z\bar w})_{(i,u)}=i\frac{3u}{4(1-f)}, \quad \det(\g_f)_{(z,w)}=\frac{16}{9}\Imm(z)^2(f')^2(1-f)^2 \ .
\end{align*}
\begin{proposition}\label{lemmascalare}
The Ricci tensor and the scalar curvature of the pseudo-K\"ahler metrics $(\g_f,\i,\ome_f)$ are given by:\begin{align*}
    &(R_{z\bar z})_{(i,u)}=\frac 1{2}+3u^2\bigg(\frac{f'}{1-f}-\frac{f''}{f'}\bigg)+\frac 9{2}u^4G_f \\ &(R_{w\bar w})_{(i,u)}=-2\bigg(\frac{f''}{f'}-\frac{f'}{1-f}\bigg)+2u^2G_f \\ &(R_{z\bar w})_{(i,u)}=i\bigg(3u\bigg(\frac{f''}{f'}-\frac{f'}{1-f}\bigg)-3u^3G_f\bigg) \\ & \normalfont{\text{scal}}(\g_f)_{(i,u)}=\frac{1}{1-f}-\frac 3{4}\frac{f''}{(f')^2}+\frac 3{2}\frac{u^2}{1-f}\bigg(6u^2G_f+\frac{11}{2}\bigg(\frac{f'}{1-f}-\frac{f''}{f'}\bigg)+\frac{G_f(1-f)}{2f'}\bigg)
\end{align*}where $\displaystyle G_f:=\frac{f''(1-f)+(f')^2}{(1-f)^2}-\frac{f'''\cdot f-(f'')^2}{(f')^2} \ .$
\end{proposition}
\begin{proof}
Using the formulae above and the symmetries $\overline{R_{i\bar j}}=R_{j\bar i}$, the Ricci tensor is given by $$\text{Ric}_{\g_f}=R_{z\bar z}\mathrm dz\otimes\mathrm d\bar z+R_{w\bar w}\mathrm dw\otimes\mathrm d\bar w+2\Ree(R_{z\bar w})\mathrm dz\otimes\mathrm d\bar w \ .$$ According to Lemma \ref{lem:scalarepseudo}, the components can be computed as $$R_{z\bar z}=-\frac{\partial^2}{\partial z\partial\bar z}\log\big(\det(\g_f)\big),\quad R_{w\bar w}=-\frac{\partial^2}{\partial w\partial\bar w}\log\big(\det(\g_f)\big),\quad R_{z\bar w}=-\frac{\partial^2}{\partial z\partial\bar w}\log\big(\det(\g_f)\big) \ .$$ 
%Since the components depend on the second derivative of the metric, the calculation must initially be done at an arbitrary point $(z,w)\in\Hyp\times\C$ and, only at the very end, it is possible to compute them at the "special" points $(i,u)$. 
Using the expression of $\det(\g_f)$ found above we get $$\log\big(\det(\g_f)\big)=\log\bigg(\frac{16}{9}\bigg)+2\log\big(\Imm(z)\big)+\log\big((f')^2\big)+2\log\big(1-f\big) \ .$$ Finally, recalling that the functions $f,f',f'',f'''$ are all evaluated at $\Imm(z)^3|w|^2$ and using the formula $\displaystyle\frac{\partial}{\partial z}\Imm(z)^l=\frac{l}{(2i)^l}\Imm(z)^{l-1}$ we obtain the desired expression for the components of the Ricci tensor.\newline The scalar curvature is given, by definition, by $$\scalg(\g_f)=\g^{z\bar z}_fR_{z\bar z}+\g^{w\bar w}_fR_{w\bar w}+\g^{z\bar w}_fR_{z\bar w}+\g^{w\bar z}_fR_{w\bar z} \ .$$ Since $\g^{w\bar z}_fR_{w\bar z}=\overline{\g^{z\bar w}_fR_{z\bar w}}$, the final expression is $$\scalg(\g_f)=\g^{z\bar z}_fR_{z\bar z}+\g^{w\bar w}_fR_{w\bar w}+2\Ree(\g^{z\bar w}_fR_{z\bar w}) \ .$$ Now, we can directly compute the scalar curvature at the points $(i,u)$. By a simple, but long enough, direct calculation, one gets the desired formula.
\end{proof}
As explained at the beginning of the paper, these expressions are too complicated to be able to make any estimates on the scalar curvature. On the other hand, on $\Hyp\times\{0\}\subset\Hyp\times\C$ the expression is considerably simplified, indeed given that $f(0)=0$, it follows that \begin{equation}
    \scalg(\g_f)_{(i,0)}=1-\frac{3}{4}\frac{f''(0)}{f'(0)^2} \ .
\end{equation}In particular, if we pick the function $f$ to be of the form $f(t)=-kt$, with $k>0$, it becomes clear that the scalar curvature on $\Hyp\times\{0\}$ is constant and equal to $1$.
\begin{corollary}\label{corscalg}
For any $(i,u)\in\Hyp\times\C^*$ and for $f(t)=-kt$, with $k>0$, the scalar curvature \normalfont{scal}$(\g_f)_{(i,u)}$ is strictly less then $1$.
\end{corollary}
\begin{proof}
For this choice of $f$, at the point $(i,u)$ we have $$f'(t)=-k, \quad f''=f'''=0,\quad G_f(t)=\frac{k^2}{(1+kt^2)},\qquad t=u^2\neq 0 \ .$$ Thanks to Proposition \ref{lemmascalare} it follows that $$\normalfont{\text{scal}}(\g_f)_{(i,u)}=\frac{1}{1+kt}\bigg(1+\frac{3}{2}t\bigg(\frac{6tk^2}{(1+kt)^2}-\frac{6k}{1+kt}\bigg)\bigg) \ .$$Using that $\displaystyle\frac{1}{1+kt}<1$ for all $k>0$ and $t>0$, we obtain $$\normalfont{\text{scal}}(\g_f)_{(i,u)}<1+\frac{9tk}{1+kt}\bigg(\frac{tk}{1+kt}-1\bigg) \ .$$The last quantity is strictly less then $1$ since $$\frac{9tk}{1+kt}>0,\quad \frac{tk}{1+kt}-1<0,\quad \forall t,k>0 \ .$$
\end{proof}

\subsection{The isometry group}
At this point we have all the tools necessary to prove the last result of the paper. It is clear from Theorem \ref{teo:rungitamburelli} that any matrix in $\PSL(2,\R)$ and any rotation of the fiber generated by $S^1$, is an isometry of $\Hyp\times\C$ with respect to $\g_f$. Plus, the two actions commute. We will show that, whenever $f(t)=-kt$, with $k>0$, any other isometry $h$, isotopic to the identity, can be written as composition $h=P\circ e^{i\theta}$ for some $(P,e^{i\theta})\in\PSL(2,\R)\times S^1$. \\ \\ We recall the following standard result, the proof of which can be found in any text on Riemannian geometry. \begin{lemma}\label{lem:1jetisometry}
Let $h_1,h_2:(M_1,g_1)\to(M_2,g_2)$ be two isometries between Riemannian smooth connected manifolds. If there is a point $p\in M_1$ such that $h_1(p)=h_2(p)$ and $\mathrm d_ph_1\equiv\mathrm d_ph_2$, then $h_1\equiv h_2$.
\end{lemma}
\begin{remark}
It is not difficult to see that the proof of this lemma can also be extended to the pseudo-Riemannian case, since only the existence of the exponential map and the connectedness hypothesis on the manifolds are used.
\end{remark}

\begin{manualtheorem}B
Let $\Isom_0\big(\Hyp\times\C, \g_f\big)$ be the connected component of the identity of the isometry group $\Isom\big(\Hyp\times\C, \g_f\big)$. If $f(t)=-kt$, with $k>0$, then $\Isom_0\big(\Hyp\times\C, \g_f\big)\cong\PSL(2,\R)\times S^1$.
\end{manualtheorem}
\begin{proof}
First notice that each isometry $h\in\Isom_0\big(\Hyp\times\C, \g_f\big)$ preserves the copy of the hyperbolic plane $\Hyp\times\{0\}$. In fact, if there was an isometry $\widetilde{h}$ with $\widetilde{h}(z,0)=(z',w)$ for some $(z',w)\in\Hyp\times\C^*$, then we would get the following contradiction \begin{align*}1&=\scalg(\g_f)_{(z,0)} \\ &=\big(\widetilde h^*\scalg(\g_f)\big)_{(z,0)} \tag{$\widetilde h$ is an isometry} \\ &=\scalg(\g_f)_{(z',w)} \\ &<1 \ . \tag{Corollary \ref{corscalg}}\end{align*} Pick any $h\in\Isom_0\big(\Hyp\times\C,\g_f \big)$ such that $h(z,0)=(z',0)$ for some $z,z'\in\Hyp$. We can always assume that $h(i,0)=(i,0)$, indeed thanks to Remark \ref{matrixSL2} there exist two matrices $P,P'\in\SL(2,\R)$ such that $(z,0)=P\cdot (i,0), \ (z',0)=P'\cdot (i,0)$, hence the isometry $(P')^{-1}\circ h\circ P$ would fix the point $(i,0)$. If we consider the linear map $\mathrm d_{(i,0)} h:T_i\Hyp\times T_0\C\to T_i\Hyp\times T_0\C$ restricted to horizontal directions, we can again assume, up to pre- and post-composition with elements in $\PSL(2,\R)$ as before, that $\mathrm d_{(i,0)} h(Z,0)=(Z,0)$, for all $Z\in T_i\Hyp$. This implies that  \begin{equation*}\mathrm d_{(i,0)} h_{|_{T_i\Hyp}}=\Id_{T_i\Hyp} \ .\end{equation*} Now if $(0,U)\in T_i\Hyp\times T_0\C^*$ is a real vertical direction, then $\mathrm d_{(i,0)} h(0,U)=(0,W)$ for some $W\in T_0\C^*$. In particular, since $h$ is an isometry we get $$\vl\vl U \vl\vl^2_{\g_f}=\vl\vl W\vl\vl^2_{\g_f} \ ,$$ which implies that  $W=e^{i\theta}U$ for some $\theta \in\R$. Furthermore, since the circular action is an isometry for $\g_f$ that is trivial on the base $\mathbb{H}^{2}$, up to pre- and post-composing with rotations in the complex plane we have $$\mathrm d_{(i,0)} h(Z,U)=(Z,U)$$ for all $Z\in T_{i}\mathbb{H}^{2}$. Since $h$ is orientation preserving, we deduce that $d_{(i,0)}h=\Id$, since $h$ should also fix an imaginary vertical tangent vector. In the end, using Lemma \ref{lem:1jetisometry}, we obtain that $h=\Id$ on the whole $\Hyp\times\C$ after possibly pre- and post-composing $h$ by elements of $\PSL(2,\R)$ and rotations. Therefore, $h$ was of the form $h=P\circ e^{i\theta}$ for some $(P,e^{i\theta}) \in \PSL(2,\R) \times S^{1}$. 
\end{proof}
During the proof of the theorem we used that each isometry isotopic to the identity preserves the orientations on both $\Hyp$ and $\C$. There are other three possibilities for an arbitrary isometry $h\in\Isom\big(\Hyp\times\C,\g_f\big)$:\begin{itemize}
    \item [(1)] $h$ reverses the orientation on $\Hyp$ and preserves the orientation on $\C$. Then, by composing with $h_1(z,w):=(-\bar z,w)$ we get an isometry preserving both orientations. Hence, the proof of Theorem \ref{thmA} holds for $h\circ h_1$. \item[(2)] $h$ preserves the orientation on $\Hyp$ and reverses the orientation on $\C$. Then, by composing with $h_2(z,w)=(z,\bar w)$ we get an isometry preserving both orientations. Hence, we have the same conclusion as above for $h\circ h_2$. \item[(3)] Finally, $h$ reverses both the orientations. Then, the same argument applies to $h\circ h_1\circ h_2$.
\end{itemize}In the end, we proved the following 
\begin{manualcorollary}C If $f(t)=-kt$, with $k>0$, then any isometry $h:\big(\Hyp\times\C,\g_f\big)\to\big(\Hyp\times\C,\g_f\big)$ can be written as $$h=P\circ e^{i\theta},\quad h=P\circ e^{i\theta}\circ h_1,\quad h=P\circ e^{i\theta}\circ h_2,\quad h=P\circ e^{i\theta}\circ h_1\circ h_2$$ for some $P\in\PSL(2,\R)$ and $e^{i\theta}\in S^1$.
\end{manualcorollary}

%Consider the local symplectomorphisms $$\widetilde\psi_{\sigma_i}:\big(T^*U/_\Lambda|_{U_i},\widetilde\omega\big)\to\big(\pi^{-1}(U_i),\omega\big)$$and consider, for each pair of indices $i,j$ such that $U_i\cap U_j\neq\empty$, the symplectomorphisms $$$$

\emergencystretch=1em

\printbibliography

\end{document}